\documentclass[11pt,english]{amsart}

\usepackage{amsmath}
\usepackage{dsfont}
\usepackage{amsfonts}
\usepackage{amsthm}
\usepackage{enumerate}
\usepackage{graphicx}
\usepackage{overpic}
\usepackage{a4wide}
\usepackage{amssymb} 
\usepackage{pictex}
\usepackage{rotating}  
\usepackage{xcolor}
\usepackage{etex}
\usepackage{cite}
\usepackage{float}

\usepackage[utf8]{inputenc}

\usepackage[all]{xy}

\numberwithin{equation}{section}

\newtheorem{theorem}{Theorem}[section]

\newtheorem{proposition}[theorem]{Proposition}
\newtheorem{lemma}[theorem]{Lemma}
\newtheorem{corollary}[theorem]{Corollary}
\newtheorem{Definition}[theorem]{Definition}
\newtheorem{notations}[theorem]{Notations}

\newenvironment{definition}{\begin{Definition}\rm}{\end{Definition}}
\newtheorem{Remark}[theorem]{Remark}
\newenvironment{remark}{\begin{Remark}\rm}{\end{Remark}}
\newtheorem{RHproblem}[theorem]{RH problem}

\newtheorem{Example}[theorem]{Example}
\newenvironment{example}{\begin{Example}\rm}{\end{Example}}

\theoremstyle{definition}

\newcommand{\indicatrice}[1]{\mathds{1}_{#1}}

\newcommand{\N}{\mathbb{N}}

\def\adots{\mathinner{\mkern2mu\raise 1pt\hbox{.}\mkern 3mu\raise 
3pt\hbox{.}\mkern1mu\raise 5pt\hbox{{.}}}}

\renewcommand{\tilde}{\widetilde}

\usepackage[bookmarksopen, naturalnames]{hyperref}

\makeatother

\begin{document}
\title[a quantitative interpretation of the frequent hypercyclicity criterion]{a quantitative interpretation of the frequent hypercyclicity criterion}
\author{R. Ernst, A. Mouze}
\address{Romuald Ernst, LMPA, Centre Universitaire de la Mi-Voix, Maison de la Recherche Blaise-Pascal, 50 rue Ferdinand Buisson, BP 699, 62228 Calais Cedex}
\email{ernst.r@math.cnrs.fr}
\address{Augustin Mouze, Laboratoire Paul Painlev\'e, UMR 8524, 
Cit\'e Scientifique, 59650 Villeneuve d'Ascq, France, Current address: \'Ecole Centrale de
Lille, Cit\'e Scientifique, CS20048, 59651 Villeneuve d'Ascq cedex}
\email{Augustin.Mouze@math.univ-lille1.fr}


\begin{abstract} We give a quantitative interpretation of the Frequent Hypercyclicity Criterion. 
Actually we show that an operator which satisfies the Frequent Hypercyclicity Criterion 
is necessarily $A$-frequently hypercyclic, where $A$ refers to some weighted densities sharper than the natural lower density. In that order, we exhibit different scales of weighted densities that are of interest to quantify the ``frequency'' measured by the Frequent Hypercyclicity Criterion. Moreover we construct an example of 
 unilateral weighted shift which is frequently hypercyclic but not $A$-frequently hypercyclic on a particular scale. 
\end{abstract}

\maketitle

\section{Introduction}\label{sec_intro} The notion of frequent hypercyclicity was introduced in the context of 
linear dynamics by Bayart and Grivaux in 2006 \cite{Baygrihyp}, \cite{Baygrifrequentlyhcop}. This latter is now a central notion in that field and is highly connected 
to combinatorics, number theory and ergodic theory. Let $X$ be a metrizable and complete topological vector space and $L(X)$ be the space of continuous linear operators on $X.$ An operator $T\in L(X)$ is 
said to be \textit{hypercyclic} if there exists $x\in X$ such that for any non-empty open set $U\subset X$, the return set 
$\{n\geq 0:T^nx\in U\}$ is non-empty or equivalently infinite. Such a vector $x$ is called a hypercyclic vector for $T$. 
Furthermore an operator $T$ is called \textit{frequently hypercyclic} if 
there exists $x\in X$ such that for any non-empty open set $U\subset X$, the set of integers 
$n$ satisfying $T^nx\in U$ has positive lower density, i.e. 
$$\liminf_{N\rightarrow +\infty}\frac{\#\{k\leq N:T^kx\in U\}}{N}>0,$$
where as usual $\#$ denotes the cardinality of the corresponding set. Thus the notion of frequent hypercyclicity 
extends the classical hypercyclicity and appraises \textit{how often} the orbit of a hypercyclic vector 
visits every non-empty open set. In the sequel we denote by $\mathbb{N}$ the set of positive integers and for 
any $x\in X$ and any subset $U\subset X$ we set $N(x,U):=\{n\in\mathbb{N}: T^nx\in U\}.$ Given a subset $E\subset\mathbb{N},$ 
we define its \textit{lower and upper} densities respectively by
$$\underline{d}(E)=\liminf_{N\rightarrow +\infty}\frac{\#\{k\leq N:k\in E\}}{N}\hbox{ and }
\overline{d}(E)=\limsup_{N\rightarrow +\infty}\frac{\#\{k\leq N:k\in E\}}{N}.$$
In other words, an operator $T\in L(X)$ is hypercyclic (resp. frequently hypercyclic) if there exists $x\in X$ such that for any non-empty open set $U\subset X,$ the set $N(x,U)$ is non-empty (resp. has positive lower density). 
To prove that an operator is hypercyclic we have at our disposal the so-called 
Hypercyclicity Criterion (see \cite{Bay}, \cite{Grope} and the references therein). In the same spirit, 
Bayart and Grivaux stated the Frequent Hypercyclicity Criterion, 
which ensures that an operator is frequently hypercyclic \cite{Baygrifrequentlyhcop}. Let 
us recall it here.
\begin{theorem}\label{frequ_hyp_crit} Let $T$ be an operator on a separable 
Fr\'echet space $X.$ If there is a dense 
subset $X_0$ of $X$ and a map $S:X_0\rightarrow X_0$ such that, for any $x\in X_0,$ 
\begin{enumerate}[(i)]
\item $\displaystyle\sum_{n=0}^{+\infty}T^nx$ converges unconditionally,
\item $\displaystyle\sum_{n=0}^{+\infty}S^nx$ converges unconditionally,
\item $TSx=x,$
\end{enumerate}
then $T$ is frequently hypercyclic. 
\end{theorem}

We already know that the above result does not characterize frequently hypercyclic operators. Indeed 
Bayart and Grivaux have exhibited a frequently hypercyclic weighted shift on $c_0$ 
that does not satisfy this criterion \cite{Baygriinv}. A natural question arises: what does the Frequent Hypercyclicity Criterion really quantify?
In order to answer this question, B\`es, Menet, Peris and Puig recently generalized 
the notion of hypercyclic operators by introducing the concept of $\mathcal{A}$-frequent hypercyclicity, where 
$\mathcal{A}$ refers to a family of subsets of $\mathbb{N}$ satisfying suitable conditions \cite{Besmenperpuig}. In particular, $\mathcal{A}$ has to satisfy the following separation condition:
\begin{align*}
&\mathcal{A} \text{ contains a sequence $(A_k)$ of disjoint sets such that for any $j\in A_k$, any $j'\in A_{k'}$, $j\neq j'$,}\\
&\text{we have $\vert j'-j\vert\geq\max(k,k')$.}
\end{align*}

In this abstract framework, they also obtain an $\mathcal{A}$-Frequent Hypercyclicity Criterion and prove that the Frequent Hypercyclicity Criterion has very strong consequences in the sense that if $T$ satisfies the Frequent Hypercyclicity Criterion, then $T$ also satisfies the $\mathcal{A}$-Frequent Hypercyclicity Criterion for any suitable family $\mathcal{A}$. 
Bonilla and Grosse-Erdmann also studied specific notions related to the concept of $\mathcal{A}$-frequent hypercyclicity in the article \cite{BongrossUFHC}.\\ 
On the other hand, the notion of frequent hypercyclicity measures the frequency and the length of the intervals when iterates of a hypercyclic vector visits every non-empty open set in a very specific way, that is given by the natural density. Actually, there are many types and notions of densities different from the natural one. Our goal is to give quantified consequences to the Frequent Hypercyclicty Criterion in terms of weighted densities.
To that purpose, we consider a special kind of lower weighted densities, generalizing the natural one but sharper than this one, 
by using the formalism of matrix summability methods. For such a matrix $A$, we use the concept of $A$-density and $A$-frequent 
hypercyclicity (see Definitions \ref{defdens} and \ref{defiA} below). 
These general densities were already used in the context of linear dynamics to study frequently universal series \cite{Mouzemupol}. 
In the present paper, we show that the Frequent Hypercyclicity Criterion gives a stronger conclusion than frequent hypercyclicity, that we quantify thanks to explicit weighted densities on different scales. 
We refer the reader to Proposition \ref{main_prop} and Theorem \ref{main_theorem} below. Therefore an operator which satisfies the Frequent Hypercyclicity Criterion 
is necessarily $A$-frequently hypercyclic, where $A$ refers to some weighted densities sharper than the natural lower density.

Let us return to the formalism of $\mathcal{A}$-frequent hypercyclicity. For instance $\mathcal{A}$ could be a family of subsets with positive given lower weighted density satisfying 
the aforementioned separation property. However from \cite{Besmenperpuig}, there is an underlying question: does there exist a frequently hypercyclic operator not being 
$\mathcal{A}$-frequently hypercyclic? We give a positive answer by constructing an unilateral weighted shift on $c_0$ which is frequently hypercyclic but not $A$-frequently hypercyclic with respect to some $A$-densities covered by the criterion (see Theorem \ref{counterexample} below).\\

The paper is organized as follows: in Section \ref{ad} we introduce some densities that will be of interest in the sequel and some properties on these densities. 
Section \ref{const_classical} is devoted to an improvement of the Frequent Hypercyclicity Criterion for a certain scale of weighted densities. In Section \ref{furth_result}, we modify this proof in order to obtain a stronger result to the criterion. Finally in Section \ref{fhc_op},
we exhibit a new example, inspired by \cite{Bayru}, of an operator which is frequently hypercyclic although it does not satisfy the Frequent 
Hypercyclicity Criterion. To ensure this latter property, we will show that this operator is not $A$-frequently hypercyclic for some suitable matrix $A.$

\section{Densities: preliminary results}\label{ad} In this section, we state some definitions and results we shall need throughout the paper. 
Let us first introduce the concept of summability matrix and its connections with some kind of densities 
on subsets of $\mathbb{N}.$

\begin{definition}{\rm A {\it summability matrix} is an infinite matrix $M=(m_{n,k})$ of complex numbers.}
\end{definition}

Let us recall that, if $(x_n)$ is a sequence and $M=(m_{n,k})$ is a summability matrix, then by $Mx$ we denote the sequence 
$((Mx)_1,(Mx)_2,\dots)$ where $(Mx)_n=\sum_{k=1}^{+\infty}m_{n,k}x_k.$ The matrix $M$ is called {\it regular} if 
the convergence of $x$ to $c$ implies the convergence of 
$Mx$ to $c.$ By a well-known result of Toeplitz (see for instance \cite{Zygmund}), 
$M$ is regular if and only if the following three conditions hold:
\begin{equation}\label{CondToeplitz}
\left\{\begin{array}{rl}
(i)&\displaystyle\lim_{n\rightarrow +\infty}m_{n,k}=0,\hbox{ for all }k\in\mathbb{N},\\
(ii)&\displaystyle\lim_{n\rightarrow +\infty}\sum_{k\geq 1}m_{n,k}=1,\\
(iii)&\sup_n \sum_{k\geq 1}\vert m_{n,k}\vert<\infty.\end{array}\right.
\end{equation}

Freedman and Sember showed that every regular summability matrix $M$ with non-negative real coefficients defines a density $\underline{d}_M$ on subsets of $\mathbb{N}$, called lower $M$-density \cite{Freedman}.

\begin{definition}\label{defdens} {\rm For a regular matrix $M=(m_{n,k})$ with non-negative coefficients 
and a set $E\subset \mathbb{N},$ the lower $M$-density 
of $E,$ denoted $\underline{d}_M(E),$ is defined by 
$$\underline{d}_M(E)=\liminf_{n\rightarrow +\infty}\sum_{k=1}^{+\infty}m_{n,k}\indicatrice{E}(k),$$
and the associated upper $M$-density, denoted by $\overline{d}_M(E),$ is defined by 
$$\overline{d}_M(E)=1-\underline{d}_M(\mathbb{N}\setminus E).$$ 
}
\end{definition} 

\begin{remark}{\rm For a non-negative regular matrix $M=(m_{n,k})$, 
Proposition 3.1 of \cite{Freedman} ensures that 
the upper $M$-density of any set $E\subset\mathbb{N}$ is given by 
$$\overline{d}_M(E)=\limsup_{n\rightarrow +\infty}\sum_{k=1}^{+\infty}m_{n,k}\indicatrice{E}(k).$$}
\end{remark}

Let $(\alpha_k)_{k\geq 1}$ 
be a non-negative sequence such that $\sum_{k=1}^n\alpha_k\rightarrow +\infty$ as $n$ tends to $\infty.$ 
Then, we deal with the special case of $A$-density where 
we write $A=\left(\alpha_k/\sum_{j=1}^n\alpha_j\right),$ 
when $A=(\alpha_{n,k})$ with $\alpha_{n,k}=\alpha_k/\sum_{j=1}^n\alpha_j$ for $1\leq k\leq n$ and $\alpha_{n,k}=0$ for $k>n.$ 
It is easy to check that $A$ is a non-negative regular summability matrix. In summability theory the transformation 
given by $x=(x_n)\mapsto Ax$ is called the Riesz mean $(A,\alpha_n).$ Here the associated $A$-density can be viewed as 
a weighted density with respect to the non-negative weight sequence $(\alpha_k)_{k\geq 1}.$ 

\begin{definition}{\rm A summability matrix $A=\left(\alpha_k/\sum_{j=1}^n\alpha_j\right)$ as above 
will be called an {\it admissible} matrix. We define its summatory function $\varphi_{\alpha}$ as follows:
$\varphi_{\alpha}: \N\rightarrow \mathbb{R}_+,$  $\varphi_{\alpha}(n)=\sum_{k\leq n}\alpha_k.$}
\end{definition}

\begin{example}\label{examplelogdensity} \begin{enumerate}{\rm 
\item \label{examplelogdensity1} If $\alpha_k=1,$ $k=1,2,\dots,$ then the summability matrix $A$ is the well-known Ces\`aro matrix and $\underline{d}_A$ is the natural lower density. 
\item \label{examplelogdensity2} If $\alpha_k=1/k,$ $k=1,2,\dots,$ $\underline{d}_A$ is the so-called lower logarithmic density, which is 
derived from the well-known logarithmic summability method. We have 
$\varphi_{\alpha}(k)\sim \log (k),$ as $k$ tends to $+\infty.$
\item \label{examplelogdensity3} The special case $\alpha_k=k^r,$ $r\geq -1$, for $k=1,2,\dots,$ 
generalizes both the natural density ($r=0$) and 
the logarithmic density ($r=-1$). Clearly we have $\varphi_{\alpha}(k)\sim \frac{k^{r+1}}{r+1},$ as $k$ tends to 
$+\infty,$ when $r>-1.$ 
\item \label{examplelogdensity4} If $\alpha_k=e^{k^r},$ $0<r< 1,$ for $k=1,2,\dots,$ an easy calculation gives $\varphi_{\alpha}(k)\sim \frac{k^{1-r}}{r}e^{k^r},$ as $k$ tends to 
$+\infty.$
\item \label{examplelogdensity5} If $\alpha_k=e^{k}$, for $k=1,2,\dots,$ then $\varphi_{\alpha}(k)\sim \frac{e}{e-1}e^{k},$ as $k$ tends to 
$+\infty.$
\item \label{examplelogdensity6} If $\alpha_1=1$ and $\alpha_k=e^{k/\log^r(k)},$ $r> 0,$ for $k=2,3,\dots,$ a summation 
by parts gives $\varphi_{\alpha}(k)\sim \log^r(k)e^{k/\log^r(k)},$ as $k$ tends to 
$+\infty.$
\item \label{examplelogdensity7} Let $h_s$ be the real function defined by $h_{s}=\log.\log^{(s)},$ 
with $\log^{(s)}=\log\circ\log\dots\circ\log,$ where $\log$ appears $s$ times. 
If $\alpha_k=e^{k/h_s(k)},$ $l\in\mathbb{N},$ $s\geq 2,$ for $k$ large enough, again a summation 
by parts gives $\varphi_{\alpha}(k)\sim h_s(k)e^{k/h_s(k)},$ as $k$ tends to 
$+\infty.$}
\end{enumerate}
\end{example}   

In the following sections, we shall use the following definitions connected to Example \ref{examplelogdensity}.

\begin{definition}\label{def1} \textrm{ We denote by 
\begin{enumerate}
\item $C_r$ the admissible matrix $C_r=\left(k^r/\sum_{j=1}^nj^r\right),$ $r\geq -1;$
\item $A_r$ the admissible matrix $A_r=\left(e^{k^r}/\sum_{j=1}^n e^{j^r}\right),$ $r\geq 0;$
\item $B_r$ the admissible matrix $B_r=(\alpha_k/\sum_{j=1}^n\alpha_j)$ with $\alpha_1=1$ and 
$\alpha_k=e^{k/\log^r(k)},$ for $k\geq 2,$ $r\geq 0$;
\item Let $h_s$ be the real function defined by $h_{s}=\log.\log^{(s)},$ 
with $\log^{(s)}=\log\circ\log\dots\circ\log,$ where $\log$ appears $s$ times. We denote by 
$\tilde{B}_s$ the admissible matrix $\tilde{B}_s=(\alpha_k/\sum_{j=1}^n\alpha_j)$ with 
$\alpha_k=e^{k/h_s(k)},$ for $k$ large enough and $s\geq 2.$ 
\end{enumerate}
}
\end{definition}

For any subset $E\subset\mathbb{N},$ we can write $E$ as a strictly increasing sequence 
$(n_k)$ of positive integers. It is well-known that $\underline{d}(E)=\liminf_{k\rightarrow +\infty}
\frac{k}{n_k}$ which allows to deduce the following simple fact: $\underline{d}(E)>0$ if and only if 
the sequence $\left(\frac{n_k}{k}\right)$ is bounded \cite{Grope}. The following lemma extends this remark to suitable $A$-densities. 

\begin{lemma}\label{LemmaDensInfCalc} Let $(\alpha_k)$  be a non-negative sequence such 
that $\sum_{k\in\N} \alpha_k=+\infty.$ Assume that the sequence $(\alpha_n/\sum_{j=1}^n\alpha_j)$ converges to 
zero as $n$ tends to $+\infty$. Let $(n_k)$ be an increasing sequence of integers forming a subset 
$E\subset\mathbb{N}.$
Then, we have 
$$\underline{d}_{A}(E)=\liminf_{k\rightarrow +\infty}\left(\frac{\sum_{j=1}^{k}\alpha_{n_j}}{\sum_{j=1}^{n_k}\alpha_j}\right),$$
where $\underline{d}_{A}$ is the $A$-density given by the summability matrix $A=(\alpha_k/\sum_{j=1}^n\alpha_j).$ 
\end{lemma}

\begin{proof} Let us consider $n_k\leq N <n_{k+1},$ then 
$$\frac{\sum_{j=1}^{k+1}\alpha_{n_j}}{\sum_{j=1}^{n_{k+1}}\alpha_j}-\frac{\alpha_{n_{k+1}}}{\sum_{j=1}^{n_{k+1}}\alpha_j}=
\frac{\sum_{j=1}^k\alpha_{n_j}}{\sum_{j=1}^{n_{k+1}}\alpha_j}
\leq 
\frac{\sum_{n_j\leq N}\alpha_{n_j}}{\sum_{j=1}^N\alpha_j}\leq
\frac{\sum_{j=1}^k\alpha_{n_j}}{\sum_{j=1}^{n_k}\alpha_j}.$$
Thus, we deduce
$$\underline{d}_{A}(E)=\liminf_{N\rightarrow +\infty}
\left(\frac{\sum_{n_j\leq N}\alpha_{n_j}}{\sum_{j=1}^N\alpha_j}\right)=
\liminf_{k\rightarrow +\infty}\left(\frac{\sum_{j=1}^{k}\alpha_{n_j}}{\sum_{j=1}^{n_k}\alpha_j}\right).$$
\end{proof}

In the present paper, we are mainly interested in sharper $A$-densities than the classical natural density. 
From this point of view, the following lemma gives some conditions to ensure that the sequence $(\alpha_k)$ leads to a sharper density.

\begin{lemma}\label{lemmastieltjes} Let $(\alpha_k)$ and $(\beta_k)$  
be non-negative sequences such that $\sum_{k\in\N} \alpha_k=\sum_{k\in\N} \beta_k=+\infty.$ Assume that the sequence $(\alpha_k/\beta_k)$ is eventually decreasing to zero. 
Let $A=(\alpha_k/\sum_{j=1}^n\alpha_j)$ and 
$B=(\beta_k/\sum_{j=1}^n\beta_j)$ be the associated admissible matrices. 
Then, for every subset $E\subset\mathbb{N},$ we have 
$$\underline{d}_{B}(E)\leq \underline{d}_{A}(E) \leq \overline{d}_{A}(E) \leq \overline{d}_{B}(E).$$
\end{lemma}

\begin{proof} Let $E$ be a subset of $\mathbb{N}.$ For every $n\geq 1,$ let us define 
$\Lambda_E^{\alpha}(n)=\sum_{k=1}^n\alpha_k \indicatrice{E}(k)$ (resp. $\Lambda_E^{\beta}(n)=\sum_{k=1}^n\beta_k \indicatrice{E}(k)$). 
In particular, one may observe that $\Lambda_{\N}^{\alpha}=\varphi_{\alpha}.$ 
Now, let $N\geq 1$ be an integer such that the sequence $(\alpha_k/\beta_k)_{k\geq N}$ is decreasing. Then for every $n\geq N+1,$ we have 
$$\begin{array}{rcl}\displaystyle\sum_{k=N+1}^n \alpha_k\indicatrice{E}(k)&=&\displaystyle
\sum_{k=N+1}^{n-1}\Lambda_E^{\beta}(k)\left(\frac{\alpha_k}{\beta_k} - 
\frac{\alpha_{k+1}}{\beta_{k+1}}\right)+\Lambda_E^{\beta}(n)\frac{\alpha_n}{\beta_n}-\Lambda_E^{\beta}(N)\frac{\alpha_{N+1}}
{\beta_{N+1}}\\
&=&\displaystyle\sum_{k=N+1}^{n-1}\frac{\Lambda_E^{\beta}(k)}
{\varphi_{\beta}(k)}\varphi_{\beta}(k)\left(\frac{\alpha_k}{\beta_k} - 
\frac{\alpha_{k+1}}{\beta_{k+1}}\right)+\frac{\Lambda_E^{\beta}(n)}{\varphi_{\beta}(n)}\varphi_{\beta}(n)
\frac{\alpha_n}{\beta_n}-\Lambda_E^{\beta}(N)\frac{\alpha_{N+1}}{\beta_{N+1}}.
\end{array}$$
Moreover, since $(\alpha_k/\beta_k)$ is a non-negative decreasing sequence and  
$\sum\alpha_k=+\infty,$ we deduce
$$\begin{array}{rcl}\displaystyle\overline{d}_{A}(E)&=&
\displaystyle\limsup_{n\rightarrow +\infty}\left[\left(\varphi_{\alpha}(n)\right)^{-1}
\left(\sum_{k=1}^N\alpha_k\indicatrice{E}(k)+\displaystyle\sum_{k=N+1}^n \alpha_k\indicatrice{E}(k)\right)\right]\\
&=&
\displaystyle\limsup_{n\rightarrow +\infty}\left[(\varphi_{\alpha}(n))^{-1}
\left(\displaystyle\sum_{k=N+1}^{n-1}\frac{\Lambda_E^{\beta}(k)}
{\varphi_{\beta}(k)}\varphi_{\beta}(k)\left(\frac{\alpha_k}{\beta_k} - 
\frac{\alpha_{k+1}}{\beta_{k+1}}\right)+\frac{\Lambda_E^{\beta}(n)}{\varphi_{\beta}(n)}\varphi_{\beta}(n)\frac{\alpha_n}{\beta_n}\right)\right]\\
&\leq &\displaystyle\sup_{k>N}\left(\frac{\Lambda_E^{\beta}(k)}{\varphi_{\beta}(k)}\right)
\limsup_{n\rightarrow +\infty}\left[(\varphi_{\alpha}(n))^{-1}\left(
\displaystyle\sum_{k=N+1}^{n-1}\varphi_{\beta}(k)\left(\frac{\alpha_k}{\beta_k} - 
\frac{\alpha_{k+1}}{\beta_{k+1}}\right)+\varphi_{\beta}(n)
\frac{\alpha_n}{\beta_n}\right)\right].
\end{array}$$
Since $\sum_{k=N+1}^{n-1}\varphi_{\beta}(k)\left(\frac{\alpha_k}{\beta_k} - 
\frac{\alpha_{k+1}}{\beta_{k+1}}\right)+\varphi_{\beta}(n)
\frac{\alpha_n}{\beta_n}=\varphi_{\beta}(N+1)\frac{\alpha_{N+1}}{\beta_{N+1}}+\sum_{k=N+2}^n\alpha_k,$ 
we get
$$\overline{d}_{A}(E)\leq \sup_{k>N}\left(\frac{\Lambda_E^{\beta}(k)}{\varphi_{\beta}(k)}\right).$$ 
Hence letting $N\rightarrow +\infty,$ we obtain 
$\overline{d}_{A}(E)\leq \overline{d}_{B}(E)$.\\
The other inequality is obtained using the relations
$\underline{d}_A(E)=1-\overline{d}_A(\mathbb{N}\setminus E)$ and 
$\underline{d}_B(E)=1-\overline{d}_B(\mathbb{N}\setminus E).$
\end{proof}

From now on, we are interested in densities given by special admissible matrices given in Definition \ref{def1}. In this case, 
Lemma \ref{lemmastieltjes} leads to the following inequalities.

\begin{lemma}\label{lemmacomp} For every subset $E\subset \mathbb{N}$ and for any $0<r\leq r',$ $0<s\leq s'<1,$ $1<t\leq t',$ $2\leq l\leq l',$ we have
$$\underline{d}_{A_1}(E)= 
\underline{d}_{B_0}(E)\leq 
\underline{d}_{\tilde{B}_{l'}}(E)\leq \underline{d}_{\tilde{B}_{l}}(E)\leq 
\underline{d}_{B_t}(E)\leq \underline{d}_{B_{t'}}(E)$$
and 
$$\underline{d}_{B_{t'}}(E)\leq 
\underline{d}_{{A}_{s'}}(E)\leq \underline{d}_{{A}_{s}}(E)\leq 
\underline{d}_{C_{r'}}(E)\leq \underline{d}_{C_{r}}(E)\leq 
\underline{d}(E).$$
\end{lemma}

Moreover, observe that a subset $E$ of $\mathbb{N}$ possesses a strictly positive natural lower density 
if and only if it has a strictly positive lower $C_r$-density, for any $r>-1.$ 

\begin{lemma}\label{lemme_poly} Let $r>-1$. Then, for every subset 
$E\subset \mathbb{N}$ the following assertions are equivalent
\begin{enumerate}[(i)]
\item\label{firstpoly} $\underline{d}_{C_r}(E)>0$,
\item\label{secondpoly} $\underline{d}(E)>0.$
\end{enumerate}
\end{lemma}

\begin{proof} Let $(n_k)\subset \mathbb{N}$ be the increasing sequence of elements of $E$.
We divide the proof in two cases.\\
\textit{Case $r\geq 0$:} 
Lemma \ref{lemmastieltjes} gives 
$\underline{d}_{C_r}(E)\leq \underline{d}(E),$ hence $(\ref{firstpoly})\Rightarrow(\ref{secondpoly}).$ For the other implication, assume that 
$\underline{d}(E)>0.$ This means that the sequence $\left(\frac{n_j}{j}\right)$ is bounded (see Lemma \ref{LemmaDensInfCalc}). 
We deduce that there exists an integer $M\geq 1$ such that for every $k\in\N$,  $k\leq n_k\leq Mk$ and
$$\sum_{j=1}^{n_k}j^r\leq \sum_{j=1}^{Mk}j^r\leq \frac{(Mk+1)^{r+1}}{r+1}\leq 
M^{r+1}\frac{(k+1)^{r+1}}{r+1}.$$
Using the fact that $\sum_{j=1}^{k}j^r\sim \frac{k^{r+1}}{r+1}\sim \frac{(k+1)^{r+1}}{r+1},$ as 
$k$ tends to $+\infty$ (cf Example \ref{examplelogdensity}) we deduce that there exists $C>0$ such that 
$$\sum_{j=1}^{n_k}j^r\leq CM^{r+1}\sum_{j=1}^{k}j^r.$$
Therefore we have 
$$\sum_{j=1}^{n_k}j^r\leq CM^{r+1}\sum_{j=1}^{k}(n_j)^r$$ and 
$\liminf_{k\rightarrow +\infty}\left(\frac{\sum_{j=1}^{k}(n_j)^r}{\sum_{j=1}^{n_k}j^r}\right)>0$ which is sufficient to conclude thanks to Lemma \ref{LemmaDensInfCalc}.\\
\textit{Case $-1<r < 0$:} As in the previous case, Lemma \ref{lemmastieltjes} gives the implication $(\ref{secondpoly})
\Rightarrow (\ref{firstpoly})$ because $\underline{d}_{C_r}(E)\geq \underline{d}(E).$ For the converse, assume that 
$\underline{d}_{C_r}(E)>0.$ According to Lemma \ref{LemmaDensInfCalc}, there exists $C>0$ such that
$$\sum_{j=1}^{n_k}j^r\leq C\sum_{j=1}^k (n_j)^r.$$ Using the inequality $j\leq n_j$  and since $r<0$, we get
$$\frac{(n_k+1)^{r+1}}{r+1}-\frac{1}{r+1}
\leq \sum_{j=1}^{n_k}j^r\leq C\sum_{j=1}^k (n_j)^r\leq 
C\sum_{j=1}^k j^r\leq
C\left(\frac{k^{r+1}}{r+1}-\frac{1}{r+1}+1\right).$$
The positivity of $r+1$ now ensures that the sequence $\left(\frac{n_k}{k}\right)$ is bounded and $\underline{d}(E)>0.$
\end{proof}

Finally let us also extend the definition of frequently hypercyclic operators using the general notion of $A$-densities introduced before. 

\begin{definition}\label{defiA}{\rm Let $A=\left(\alpha_k/\sum_{j=1}^n\alpha_j\right)$ be an admissible matrix. 
Using the notations of Section \ref{sec_intro}, an operator $T\in L(X)$ is said to be $A$-frequently hypercyclic if 
there exists $x\in X$ such that for any non-empty open set $U\subset X,$ the set $N(x,U)$ has positive lower $A$-density.
}
\end{definition}

\section{Frequent hypercyclicity criterion: classical construction}\label{const_classical} According to Lemma 
\ref{lemme_poly}, a frequently hypercyclic operator is necessarily $C_r$-frequently hypercyclic for any $r>-1.$ 
So, even if this is obvious, the Frequent Hypercyclicity Criterion allows to obtain the $B_r$-frequent hypercyclicity too. A careful examination 
of the well-known proof of this criterion leads to a more precise result. Indeed, in the classical proof of the Frequent Hypercyclicity Criterion 
the following constructive lemma plays a prominent role \cite[Lemma 9.5]{Grope}. 

\begin{lemma}\label{Lemgrope} There exist pairwise disjoint subsets $A(l,\nu),$ $l,\nu\geq 1,$ of 
$\mathbb{N}$ of positive lower density such that, for any $n\in A(l,\nu)$ and $m\in A(k,\mu)$, we 
have that $n\geq \nu$ and 
$$\vert n-m\vert\geq \nu+\mu\hbox{ if }n\neq m.$$
\end{lemma}

The proof of this result is based on a specific partition of $\mathbb{N}$ using the dyadic representation 
$n=\sum_{j=0}^{+\infty}a_j2^j=(a_0,a_1,\dots)$ of any positive integer. Actually the authors define the sets 
$I(l,\nu),$ $l,\nu\geq 1,$ as the sets of all $n\in\mathbb{N}$ whose dyadic representation has the form 
$n=(0,\dots,0,1,\dots,1,0,*)$ with $l-1$ leading zeros, exactly followed by $\nu$ ones, then one zero and an 
arbitrary remainder. Let $\delta_k=\nu,$ if $k\in I(l,\nu)$ for some $l\geq 1.$  
Then they construct the following strictly increasing sequence $(n_k)$ of positive integers by setting 
$$n_k=2\sum_{i=1}^{k-1}\delta_i +\delta_k,\quad k\geq 1.$$
This construction clearly ensures that for any integers $i,j,$ with 
$i\ne j,$ the separation condition stated in Lemma \ref{Lemgrope} holds, that is 
$$\vert n_i-n_j\vert\geq \delta_i+\delta_j.$$
Finally they define the sets $A(l,\nu)=\{n_k;k\in I(l,\nu)\}$ and they prove that these sets have positive lower density since $(n_k)$ does and the sets $I(l,\nu)$ 
are arithmetic sequences. Actually we are going to prove that the sequence $(n_k)$ has positive lower $B_2$-density. To do that, we start by giving an exact formula for this sequence 
that will allow to obtain easily its asymptotic behavior. We obtain the following result, whose proof will be given later (see Lemma \ref{retrouve_lemme} below). 

\begin{lemma}\label{lemmaestimgre}
If $k=2^n+\sum_{i=0}^{n-1}\alpha_i2^i$ with $\alpha_i\in\{0;1\}$ for every $0\leq i\leq n-1$, then
$$n_k=4k-2\left(\sum_{i\in I_k} L_i(i+1)\right)-\delta_k,$$
where $I_k$ stands for the set of integers $i$ such that $\alpha_i$ is the first non-zero integer of a block 
(of consecutive non-zero coefficients) having length $L_i$ in the dyadic decomposition of $k$.
\end{lemma}

From this lemma, we deduce the following estimate using the same notations. 

\begin{proposition}\label{prop_dens_b2} The sequence $(n_k)$ satisfies the following estimate $n_k-4k=O(\log ^2(k)).$ Moreover 
this estimate is optimal in the following sense: there exists an increasing sub-sequence $(\lambda_n)$ of positive integers such that the sequence $\frac{n_{\lambda_j}-4\lambda_j}{\log^2 (\lambda_j)}$ converges to a non-zero real number. 
\end{proposition}

\begin{proof} According to Lemma \ref{lemmaestimgre}, we have, for $k=2^l+\sum_{i=0}^{l-1}\alpha_i2^i$ with $\alpha_i\in\{0;1\}$ for every $0\leq i\leq l-1$, 
$$n_k=4k-2\left(\sum_{i \in I_k}L_i (i+1)\right)-\delta_k,$$
where $I_k=\{i\in\mathbb{N};\ \alpha_{i}\ne 0\hbox{ and }\alpha_{i-1}=0\},$ with the conventions $\alpha_{-1}=0$, $\alpha_{l}=1$ and for $i\in I_k,$ 
$L_i=\min\{j; \alpha_{i+j}=0\}.$ 
Obviously we deduce
$$n_k\leq 4k-2\log_2(k)-1.$$ 
Notice that we have the equality $n_k=4k-2(\log_2(k)+1)-1$ for $k=2^l.$ 
\\On the other hand, we can write 
$$\sum_{i\in I_k}L_i (i+1)=\sum_{i_1<i_2<\dots <i_{m_k}}L_{i_j} (i_{j}+1),$$
where $I_k=\{i_1<i_2<\dots <i_{m_k}\}.$ Observe that we have
$$i_n+L_{i_n}+1\leq i_{n+1}\hbox{ for }n=1,\dots,m_{k}-1\hbox{ and }L_{i_{m_{k}}}=l-i_{m_k}+1.$$ 
Since $i_{m_k}\leq l,$ we get
$$\sum_{i_1<i_2<\dots <i_{m_k}}L_{i_j} (i_{j}+1)\leq 
\left(\sum_{j=1}^{m_{k}-1}(i_{j+1}-(i_j+1))(i_j+1)\right)+(l+1-i_{m_k})(i_{m_k}+1)\leq (l+1)^2.$$ 
By construction we have 
$$\delta_k\leq \log_2(k)+1.$$
Since $\log_2(k)\leq l\leq \log_2(k)+1,$ we conclude 
$$ 4k-2(\log_2(k) +2)^2- \log_2(k)-1\leq n_k\leq 4k-2\log_2(k)-1$$
and the estimate $n_k-4k=O(\log^2k)$ holds. 
Finally let us consider $\lambda_j=\sum_{l=0}^j2^{2l}.$ An easy calculation gives 
$$n_{\lambda_j}=4\lambda_j-2\sum_{l=0}^j(2l+1)-1=4\lambda_j-2j^2-4j-3.$$
Since $\lambda_j=(4^{j+1}-1)/3,$ the sequence $\left(\frac{n_{\lambda_j}-4\lambda_j}{\log^2(\lambda_j)}\right)$ converges to a non-zero real number.
\end{proof}

We now prove that the sequence $(n_k)$ constructed above not only has positive lower density but has also positive lower 
$B_2$-density.

\begin{lemma} \label{lemma_b2} We have $\underline{d}_{B_2}((n_k))>0.$

\end{lemma}
\begin{proof} Using (\ref{examplelogdensity6}) from Example \ref{examplelogdensity}, we have 
$$\underline{d}_{B_2}(n_k)=\liminf_{k\rightarrow +\infty}\left(\frac{\sum_{j=1}^{k}e^{n_j/\log^{2}(n_j)}}
{\log^{2}(n_k)e^{n_k/\log^{2}(n_k)}}\right).$$ 
According to Proposition \ref{prop_dens_b2}, there exists a constant $C>0$ such that, for $N$ large enough, 
$$\frac{\sum_{j=N}^{k}e^{n_j/\log^{2}(n_j)}}
{\log^{2}(n_k)e^{n_k/\log^{2}(n_k)}}\geq \frac{\sum_{j=N}^{k}e^{(4j-C\log^2(j))/\log^{2}(4j-C\log^2(j))}}
{\log^{2}(4k)e^{4k/\log^{2}(4k)}}.$$
A summation by parts gives
$$\sum_{j=N}^{k}e^{(4j-C\log^2(j))/\log^{2}(4j-C\log^2(j))}\sim \frac{\log^2(k)}{4}
e^{(4k-C\log^2(k))/\log^{2}(4k-C\log^2(k))},\hbox{ as }k\rightarrow +\infty.$$ 
Finally, a similar computation as those needed for Example \ref{examplelogdensity}, yields 
$$\frac{\sum_{j=N}^{k}e^{(4j-C\log^2(j))/\log^{2}(4j-C\log^2(j))}}
{\log^{2}(4k)e^{4k/\log^{2}(4k)}}\sim \frac{e^{-C}}{4},\hbox{ as }k\rightarrow +\infty,$$ 
which finishes the proof.
\end{proof}

Lemma \ref{lemma_b2} allows us to show that the Frequent Hypercyclicity Criterion gives a strengthened result.

\begin{proposition}\label{main_prop} Let $T$ be an operator on a separable Fr\'echet space $X.$ If there is a dense 
subset $X_0$ of $X$ and a map $S:X_0\rightarrow X_0$ such that, for any $x\in X_0,$ 
\begin{enumerate}[(i)]
\item $\displaystyle\sum_{n=0}^{+\infty}T^nx$ converges unconditionally,
\item $\displaystyle\sum_{n=0}^{+\infty}S^nx$ converges unconditionally,
\item $TSx=x,$
\end{enumerate}
then $T$ is $B_2$-frequently hypercyclic. 
\end{proposition}
The proof of this result is the same as the classical proof of the Frequent Hypercyclicity Criterion. Indeed, from 
Lemma \ref{lemma_b2}, we can deduce that the sets $A(l,\nu)$ not only have positive lower density but even have positive lower $B_2$-density. 
We won't detail the proof here because we will prove a stronger result in Section \ref{furth_result}.\\

Thanks to Lemma \ref{lemmacomp}, one may actually deduce the following corollary proving that the scale defined by matrices $A_r$ is not fine enough to exhibit the limit in term of densities of the Frequent Hypercyclicity Criterion. 

\begin{corollary}\label{coro_ar}
Under the assumptions of the previous proposition, the operator $T$ is $A_r$-frequently hypercyclic for every $0\leq r<1$.
\end{corollary}

The previous result proves that for any $0\leq r<1$, the $A_r$-frequent hypercyclicity phenomenon exists and is even common. On the other hand, 
one may also notice that the geometric rate of growth (i.e. $r=1$) is unreachable in terms of dynamics. More precisely, we have the following result.

\begin{proposition}
There is no $A_1$-frequently hypercyclic operator.
\end{proposition}

\begin{proof} We argue by contradiction. Assume that $T$ is a $A_1$-frequently hypercyclic operator on a Banach space $X$ and $x$ is a $A_1$-frequently hypercyclic vector. Let also $U$ be a non-empty open subset in $X$. Then, by definition and with Example \ref{examplelogdensity}, we get
$$0<\liminf_{N\to+\infty}\sum_{k\leq N}\frac{e^k}{\sum_{j=1}^{N}e^j}\indicatrice{N(x,U)}(k)=\liminf_{N\to+\infty}(1-e^{-1})\sum_{k\leq N}e^{k-N}\indicatrice{N(x,U)}(k).$$
Moreover one may remark that asserting that this limit is non-zero implies that the set $N(x,U)$ has bounded gaps. Indeed, if one suppose that $N(x,U)$ has unbounded gaps then 
there exists a sequence $(N_i)$ and a sequence $(p_i)$ tending to $+\infty$ such that for every $i\in\N$, $\{N_i-p_i+1;N_i-p_i+2;\ldots;N_i\}\cap N(x,U)=\emptyset$. This gives
$$0<\liminf_{i\to+\infty}(1-e^{-1})\sum_{k\leq N_i}e^{k-N_i}\indicatrice{N(x,U)}(k)\leq \liminf_{i\to+\infty}(1-e^{-1})\sum_{k\leq N_i-p_i}e^{k-N_i}=\lim_{i\to+\infty}e^{-p_i}-1=0$$
and this contradiction shows that the set $N(x,U)$ has bounded gaps. Let us denote by $M$ an upper bound of the length of these gaps. It suffices to choose $V$ so far from the origin such that the norm of $T$ forbids $T^k(U)$ from intersecting $V$ for $k\leq M$. This means that the orbit of $x$ will never reach the open set $V$ contradicting the $A_1$-frequent hypercyclicity of $x$.
\end{proof}

\noindent On the other hand, observe that the following result holds. 

\begin{lemma}\label{lemme_dens_br} For every $0<r<1,$ $\underline{d}_{B_r}(n_k)=0.$
\end{lemma}

\begin{proof}
We have to estimate the following limit 
$$\underline{d}_{B_r}(n_k)=\liminf_{k\rightarrow +\infty}\left(\frac{\sum_{j=1}^{k}e^{n_j/\log^{r}(n_j)}}
{\sum_{j=1}^{n_{k+1}-1}e^{j/\log^{r}(j)}}\right).$$ 

Remark that by definition of $\delta_k$, there exists an increasing sequence of integers $(\lambda_k)_{k\in\N}$ such that $n_{\lambda_k+1}-n_{\lambda_k}-1=\delta_{\lambda_{k}}
=k+1\sim\log_2(\lambda_k),$ as $k$ tends to $\infty$ (consider for example $\lambda_k=2^{k+1}-1$). 
Then,
$$\underline{d}_{B_r}(n_k)\leq\liminf_{k\rightarrow +\infty}\left(\frac{\sum_{j=1}^{\lambda_k}e^{n_j/\log^{r}(n_j)}}
{\sum_{j=1}^{n_{\lambda_k+1}-1}e^{j/\log^{r}(j)}}\right)\leq\liminf_{k\rightarrow +\infty}\left(\frac{\sum_{j=1}^{n_{\lambda_{k}}}e^{j/\log^{r}(j)}}
{\sum_{j=1}^{n_{\lambda_k+1}-1}e^{j/\log^{r}(j)}}\right).$$
Using the estimate (\ref{examplelogdensity6}) from Example \ref{examplelogdensity}  and the one from Proposition \ref{prop_dens_b2}, we get:
\begin{align*}
\underline{d}_{B_r}(n_k)&\leq\liminf_{k\rightarrow +\infty}\left(\frac{\log^r(n_{\lambda_{k}})e^{n_{\lambda_k}/\log^{r}(n_{\lambda_k})}}
{\log^{r}(n_{\lambda_k+1}-1)e^{n_{\lambda_k}/\log^{r}(n_{\lambda_k})}}\right)&\\
&\leq\liminf_{k\rightarrow +\infty}e^{n_{\lambda_k}/\log^{r}(n_{\lambda_k})-(n_{\lambda_k}+\delta_{\lambda_k})/\log^{r}(n_{\lambda_k}+\delta_{\lambda_k})}&
\end{align*}

We begin by studying the term in the exponent
$$\frac{n_{\lambda_k}}{\log^{r}(n_{\lambda_k})}-\frac{n_{\lambda_k}+\delta_{\lambda_k}}{\log^{r}(n_{\lambda_k}+\delta_{\lambda_k})}=\frac{n_{\lambda_k}}{\log^r\left(n_{\lambda_k}+\delta_{\lambda_k}\right)}\left(\left(1+\frac{\log\left(1+\frac{\delta_{\lambda_k}}{n_{\lambda_k}}\right)}{\log(n_{\lambda_k})} \right)^r-\left(1+\frac{\delta_{\lambda_k}}{n_{\lambda_k}}\right)\right)$$
which reduces to the following thanks to a Taylor expansion:
$$\frac{n_{\lambda_k}}{\log^r\left(n_{\lambda_k}+\delta_{\lambda_k}\right)}\left(\frac{1}{\log(n_{\lambda_k})}-1\right)+o\left(\frac{\delta_{\lambda_k}}{\log^{1+r}(n_{\lambda_k})}\right).$$
Now combining the estimate $\delta_{\lambda_k}\sim\log_2(\lambda_k)$ as $k$ tends to $\infty$ with the one given by Proposition \ref{prop_dens_b2}, we deduce that 
$\delta_{\lambda_k}/\log^{1+r}(n_{\lambda_k}){\rightarrow}0,$ as $k$ tends to $\infty.$ Hence we get
$$e^{n_{\lambda_k}/\log^{r}(n_{\lambda_k})-(n_{\lambda_k}+\delta_{\lambda_k})/\log^{r}(n_{\lambda_k}+\delta_{\lambda_k})}\sim 
e^{\frac{n_{\lambda_k}}{\log^r\left(n_{\lambda_k}+\delta_{\lambda_k}\right)}\left(\frac{1}{\log(n_{\lambda_k})}-1\right)}\underset{k\to+\infty}{\longrightarrow}0.$$
This proves that $\underline{d}_{B_r}(n_k)=0$.
\end{proof}

Notice that Proposition \ref{prop_dens_b2} combined with Lemma \ref{lemme_dens_br} do not allow us to conclude to the $B_r$-frequent hypercyclicity or not in the 
Frequent Hypercyclicity Criterion for $1\leq r<2.$ 

\section{Further results}\label{furth_result} In this section, we are going to improve the conclusion of the Frequent Hypercyclicity Criterion given by Proposition 
\ref{main_prop}. To do this, we will modify the sequence $(n_k)$ used in the proof of Lemma \ref{Lemgrope} to obtain a new sequence possessing a positive $A$-density for an admissible matrix $A$ defining a sharper density than the natural density. \\

Throughout this section, $(a_n)$ will be an increasing sequence of positive integers with $a_1=1$. Using this sequence we define the function $f:\mathbb{N}\rightarrow\mathbb{N},$ by 
$f(j)=m$ for all $j\in\{a_m,\ldots,a_{m+1}-1\}.$ In the spirit of the sequence studied in the previous section, we also define the sequence $(n_k(f))$ by induction: 
$$n_1(f)=f(1)=1\hbox{ and }n_{k}(f)=n_{k-1}(f)+f(\delta_{k-1})+f(\delta_{k})\hbox{ for }k\geq 2.$$ Clearly we obtain the following equality, for all $k\geq 2,$
\begin{equation}\label{Eqnkmod}
n_{k}(f)=2\sum_{i=1}^{k-1}f(\delta_{i})+f(\delta_{k}).
\end{equation}
Let us notice that, if we set $a_m=m$ for every $m\geq 1,$ then the corresponding sequence 
$(n_k(f))$ is the sequence $(n_k)$ of Section \ref{const_classical}. From now on, we will omit the notation $f$ in 
$(n_k(f))$ for sake of readability. Our purpose is to compute an exact formula for the new sequence $(n_k)$ to understand its asymptotic behavior. First of all, 
we obtain an expression for the subsequence $(n_{2^{a_m}}).$ 

\begin{lemma}\label{lemfromuleam}
For all $m\in\N,$ we have 
$$n_{2^{a_m}}=2^{a_m+1}\sum_{i=1}^{m}\frac{1}{2^{a_i-1}}-2m+1.$$
\end{lemma}

\begin{proof} Set $\Delta_j^{(m)}=\{0\leq l\leq 2^{a_m}-1:\ \delta_l=j\}.$ 
First let us observe that we have, by definition, for every $1\leq j\leq a_m$, 
$$n_{2^{a_m}}=2\sum_{k=1}^{2^{a_m}-1}f(\delta_k)+f(\delta_{2^{a_m}})=2\sum_{j=1}^{a_m}f(j)\#\Delta_j^{(m)}+f(\delta_{2^{a_m}}).$$
Thus it suffices to compute the cardinal of the set $\Delta_j^{(m)}$.
It easily follows that $\#\Delta_j^{(m)}=1+\sum_{i=0}^{a_m-j-1}2^{a_m-j-i}$. Indeed, we separate the case when the first block of ones in the dyadic decomposition of $l$ ends on $2^{a_m-1}$ and the case when the first block of ones ends before. In the first case, we have no choice, there is only one possibility but in the second case we have a certain number $i$ of zeros at the beginning, then the first block of ones, which is of length $j,$ then one zero (because the first block of ones has to be of length $j$) and then we have $2^{a_m-j-i}$ possible choices as shown below.
$$(\overbrace{\underbrace{0,0,\ldots,0}_{\text{length } i},\underbrace{1,1,\ldots,1,1,0}_{\text{length } j+1},\star,\star,\ldots,\star,\star}^{\text{length } a_m+1},0,0,\ldots).$$ 
A quick calculation leads to $\#\Delta_j^{(m)}=1+\sum_{i=0}^{a_m-j-1}2^{a_m-j-i}=2^{a_m-j}.$ Therefore we get 
$$n_{2^{a_m}}=2\sum_{j=1}^{a_m}f(j)2^{a_m-j}+f(\delta_{2^{a_m}})=2\sum_{j=1}^{a_m} f(j)2^{a_m-j}+1.$$
Now we use the link between the values of $f(j)$ and the position of $j$ compared to the sequence $(a_m)$ to compute the sum:
\begin{align*}
n_{2^{a_m}}&=2\sum_{j=1}^{a_m} f(j)2^{a_m-j}+1&\\
&=2^{a_m+1}\sum_{j=1}^{a_m-1} f(j)2^{-j}+2m+1.
\end{align*}
Let us now split the sum according to the values of $f(j)$:
\begin{align*}
n_{2^{a_m}}&=2^{a_m+1}\sum_{i=1}^{m-1} \left(\sum_{j=a_i}^{a_{i+1}-1}f(j)2^{-j}\right)+2m+1&\\
&=2^{a_m+1}\sum_{i=1}^{m-1} i\left(\sum_{j=a_i}^{a_{i+1}-1}2^{-j}\right)+2m+1&\\
&=2^{a_m+1}\sum_{i=1}^{m-1} i\left(\frac{1}{2^{a_i-1}}-\frac{1}{2^{a_{i+1}-1}}\right)+2m+1&\\
&=2^{a_m+1}\sum_{i=1}^{m} \frac{1}{2^{a_i-1}}-2m+1.
\end{align*}
\end{proof}

\noindent We strengthen the previous lemma as follows. 

\begin{lemma}\label{lemformuleamq}
Let $m\in\N.$ For every $q\in\N$ such that $q<a_m-a_{m-1},$ the following equality holds
$$n_{2^{a_m-q}}=2^{a_m-q+1}\sum_{j=1}^{m-1}\frac{1}{2^{a_j-1}}-2(m-1)+1.$$
\end{lemma}

\begin{proof} 
This proof works along the same lines as the proof of Lemma \ref{lemfromuleam}. Thus, we adapt the preceding proof. It yields 
\begin{align*}
n_{2^{a_m-q}}&=2\sum_{j=1}^{2^{a_{m}-q}-1} f(\delta_j)+f(\delta_{2^{a_m-q}})&\\
&=2\sum_{j=1}^{a_m-q} f(j)2^{a_m-q-j}+1&\\
&=2^{a_m-q+1}\left(\left(\sum_{i=1}^{m-2} \sum_{j=a_i}^{a_{i+1}-1}i2^{-j}\right)+\sum_{j=a_{m-1}}^{a_m-q}(m-1)2^{-j}\right)+1&\\
&=2^{a_m-q+1}\sum_{j=1}^{m-1}\frac{1}{2^{a_j-1}}-2(m-1)+1.
\end{align*}

\end{proof}

\begin{lemma}\label{lem2nfdec} 
For $k=2^n+\sum_{i=0}^{n-2}\alpha_i2^i$ with $\alpha_i\in\{0;1\},$ $0\leq i\leq n-2$, and 
$(\alpha_0,\ldots,\alpha_{n-2})\ne (0,\ldots,0),$ we have 

$$n_{k-2^n}=2\sum_{i=2^n+1}^{k-1}f(\delta_i)+f(\delta_k).$$
For $k=2^n+2^{n-1}+\sum_{i=0}^{n-2}\alpha_i2^i$ with $\alpha_i\in\{0;1\},$ $0\leq i\leq n-2$, we have 
$$n_{k-2^n}=2\sum_{i=2^n+1}^{k-1}f(\delta_i)-f(\delta_{k-2^n})-2(f(L)-1)+2f(\delta_k),$$
where $L$ is the length of the block of one's containing the coefficient one of $2^n$ in the dyadic decomposition of $k$.
\end{lemma}

\begin{proof} We begin by proving the first assertion. We have 
$$n_{k-2^n}=2\sum_{i=1}^{k-2^n-1}f(\delta_i)+f(\delta_{k-2^n}).$$
Since $k=2^n+\sum_{i=0}^{n-2}\alpha_i2^i,$ observe that for all $2^{n}+1\leq i\leq k-1$ the dyadic decomposition 
of $i$ contains a one for some $2^l$ with $0\leq l\leq n-2$ and the coefficient of $2^{n-1}$ is zero. 
Therefore, the first block of ones in the dyadic decomposition of $i$ does not contain 
the coefficient of $2^n$. Thus, for every such $2^n+1\leq i\leq k-1$, we have $\delta_i=\delta_{i-2^n}$. This proves the first part of the lemma since $\delta_k=\delta_{k-2^n}$.
\smallskip

To prove the second assertion, we begin by observing that 
$$2\sum_{i=2^n+1}^{k-1}f(\delta_i)=2\sum_{i=2^n+1}^{k}f(\delta_i)-2f(\delta_k).$$ 
Then, as above if the index $i$ is such that the coefficient of $2^n$ does not belong to the first block of ones (in the dyadic decomposition of $i$) then $\delta_i=\delta_{i-2^n}$.
On the other hand, if the coefficient of $2^n$ belongs to the first block of ones and 
since we have $a_{n-1}=1$ then $i$ has to be of the form $i=\sum_{l=0}^{p}2^{n-l}$ for $p\in\{1,\ldots,L-1\}$ and $\delta_i=\delta_{i-2^n}+1$. Now let us pick a particular index $i$ of the form 
$i=\sum_{l=0}^{p}2^{n-l}$ with $p\in\{1,\ldots,L-1\}.$ We consider two cases:\\
\noindent Case 1: for every $j\in\{2,\ldots,f(L)\}$, we have $p+1\ne a_j$. Then 
we have $f(\delta_i)=f(p+1)=f(\delta_i-1)=f(\delta_{i-2^n}).$\\
\noindent Case 2: there exists an integer $j\in\{2,\ldots,f(L)\}$ with $p+1= a_j.$ Then 
we get $f(\delta_i)=f(a_j)=j=f(\delta_i-1)+1=f(\delta_{i-2^n})+1$.\\
Finally we deduce
\begin{align*}
2\sum_{i=2^n+1}^{k-1}f(\delta_i)=2\sum_{i=2^n+1}^{k}f(\delta_i)-2f(\delta_k)&=
2\left(\sum_{i=1}^{k-2^n}f(\delta_i)+(f(L)-1)\right)-2f(\delta_k)&\\
&=n_{k-2^n}+f(\delta_{k-2^n})+2(f(L)-1)-2f(\delta_k).&
\end{align*}
\end{proof}

From Lemma \ref{lem2nfdec} we deduce the following result.

\begin{lemma}\label{lemdec1}
Let $L$ be any non-zero integer and $q$ be an integer. If $k=\sum_{j=0}^{L-1}2^{q+j}+k'$ with $0\leq k'<2^{q-1}$ then either $k'\ne 0$ and
$$n_k=n_{k'}+\sum_{j=0}^{L-1}n_{2^{q+j}}+2\sum_{j=2}^{L}(f(j)-1)+L,$$
or $k'=0$ and
$$n_k=\sum_{j=0}^{L-1}n_{2^{q+j}}+2\sum_{j=2}^{L}(f(j)-1)+L-f(L).$$
\end{lemma}

\begin{proof}
We proceed by induction on $L.$ 
For $L=1$, set $k=2^q+k'$. First, observe that if $k'=0$, the result is clear by Lemma \ref{lem2nfdec}. So assume that $0<k'<2^{q-1}.$ We divide $n_k$ into two sums
$$n_k=\left(2\sum_{i=1}^{2^q-1}f(\delta_i)+f(\delta_{2^q})\right)+\left(f(\delta_{2^q})+2\sum_{i=2^q+1}^{k-1}f(\delta_i)+f(\delta_{k})\right).$$
It suffices to apply Lemma \ref{lem2nfdec} to obtain 
$$n_k=n_{2^q}+f(\delta_{2^q})+n_{k'}=n_{2^q}+1+n_{k'}$$ 
and we have the desired conclusion. Now choose $L\geq 2$ and suppose that the result holds for every integer $l,$ with $1\leq l\leq L-1.$ 
By Lemma \ref{lem2nfdec}, we get
\begin{align*}
n_k&=\sum_{i=1}^{k-1}f(\delta_i)+f(\delta_k) &\\
&=\sum_{i=1}^{2^{q+L-1}-1}f(\delta_i)+2f(\delta_{2^{q+L-1}})+2\sum_{i=2^{q+L-1}+1}^{k-1}f(\delta_i)+f(\delta_{k})&\\
&=n_{2^{q+L-1}}+f(\delta_{2^{q+L-1}})+n_{k-2^{q+L-1}}+f(\delta_{k-2^{q+L-1}})+2(f(L)-1)-f(\delta_k).
\end{align*}
We have $f(\delta_{2^{q+L-1}})=1.$ Moreover suppose that $0<k'<2^{q-1}$, then the block of ones containing the coefficient 
one of $2^{q+L-1}$ is not the first one, thus $\delta_{k}=\delta_{k'}=\delta_{k-2^{q+L-1}}.$ Using the induction hypothesis, we obtain
\begin{align*}
n_k&=n_{2^{q+L-1}}+n_{k-2^{q+L-1}}+2(f(L)-1)+1&\\
&=n_{k'}+\sum_{j=0}^{L-2}n_{2^{q+j}}+2\sum_{j=2}^{L-1}(f(j)-1)+L-1+n_{2^{q+L-1}}+2(f(L)-1)+1&\\
&=n_{k'}+\sum_{j=0}^{L-1}n_{2^{q+j}}+2\sum_{j=2}^{L}(f(j)-1)+L.
\end{align*}
On the other hand, in the case $k'=0$, the induction hypothesis gives
\begin{align*}
n_k&=n_{2^{q+L-1}}+1+n_{k-2^{q+L-1}}+f(L-1)+2(f(L)-1)-f(L)&\\
&=\sum_{j=0}^{L-2}n_{2^{q+j}}+2\sum_{j=2}^{L-1}(f(j)-1)+(L-1)-f(L-1)\\&\quad\quad +n_{2^{q+L-1}}+1+f(L-1)+2(f(L)-1)-f(L)&\\
&=\sum_{j=0}^{L-1}n_{2^{q+j}}+2\sum_{j=2}^{L}(f(j)-1)+L-f(L).
\end{align*}
This completes the proof. 
\end{proof}

From Lemma \ref{lemdec1}, we immediately get the following result since $f(1)=1$.

\begin{lemma}\label{lemdecsum}  Let $L_1,\ldots, L_r$ be non-zero integers and $q_1,\ldots,q_r$ be integers such that 
$q_i+L_i<q_{i+1}$ for every $1\leq i\leq r-1$. For $k=\sum_{i=1}^{r}\sum_{j=0}^{L_i-1}2^{q_i+j}$ 
we have
$$n_k=\sum_{i=1}^{r}\sum_{j=0}^{L_i-1}n_{2^{q_i+j}}+2\sum_{i=1}^{r}\sum_{j=1}^{L_i}f(j)-\sum_{i=1}^{r}L_i-f(L_1).$$
\end{lemma}

We are ready to obtain a general formula for the sequence $(n_k).$ Let us introduce some notations. 
\begin{notations}\label{nota_sequences}{\rm  Let $L_1,\ldots, L_r$ be non-zero integers and $q_1,\ldots,q_r$ be integers such that $q_i+L_i<q_{i+1}$ for every $1\leq i\leq r-1$. 
We define the integers $m_i,t_i,s_i$ and $p_i$ as follows:
\begin{enumerate}
\item $m_i$ is the greatest integer such that $a_{m_i}\leq q_i+L_i-1,$
\item $t_i=q_i+L_i-1-a_{m_i},$
\item $p_i=\#\{l\in\mathbb{N}: l<m_i\hbox{ and } q_i\leq a_l\leq q_i+L_i-1\},$
\item $s_i=L_i-1-t_i-(a_{m_i}-a_{m_i-p_i}).$
\end{enumerate}
}
\end{notations}
\noindent To understand these notations, we give the following representation.

\vspace{-2.8cm}

\includegraphics[width=14cm]{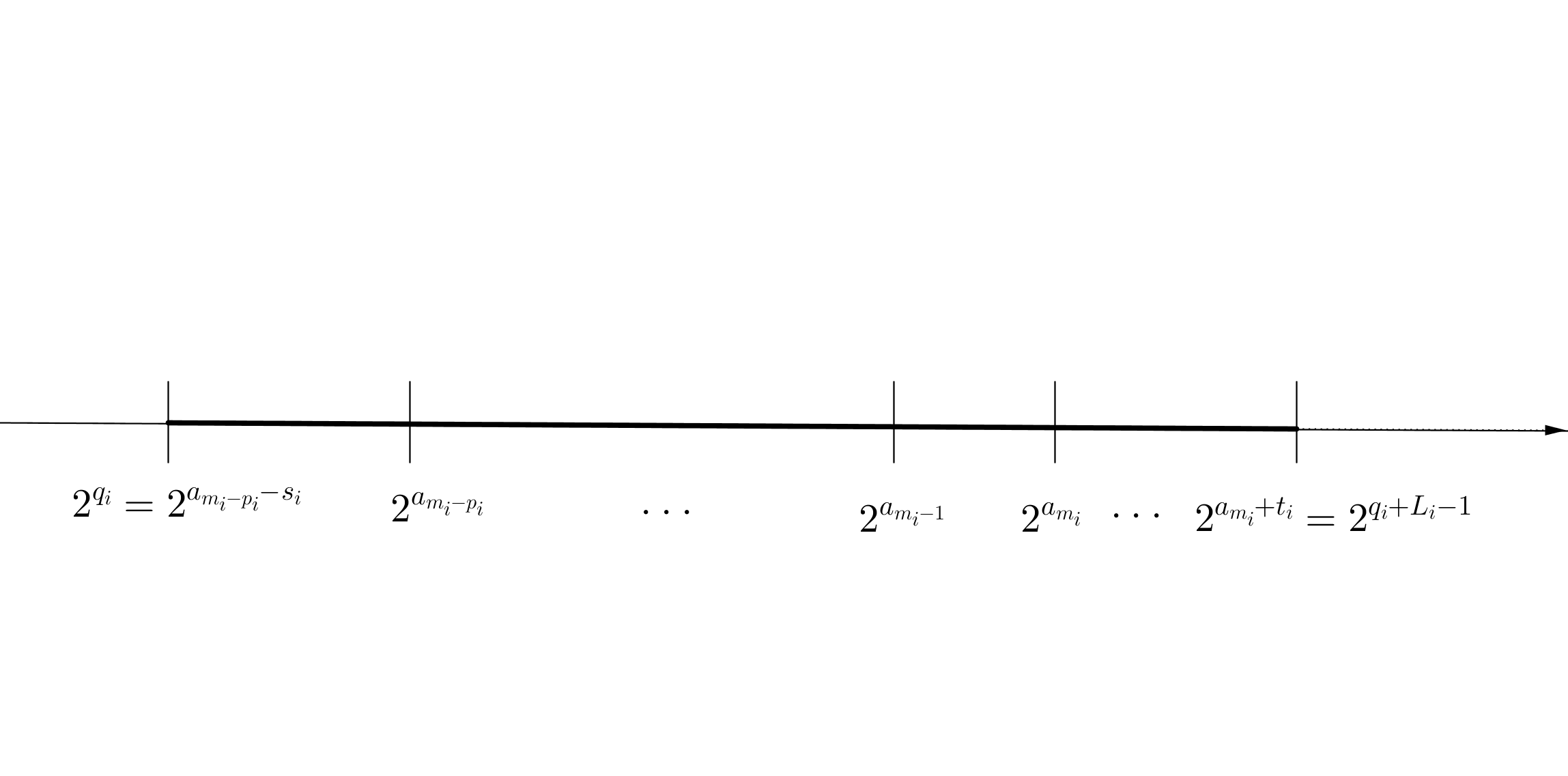}

\vspace{-1.5cm}

\noindent Now we state an explicit formula for the sequence $(n_k(f)).$ This will allow us to obtain a good asymptotic formula for this sequence.

\begin{lemma}\label{lemsumn} Using the notations (\ref{nota_sequences}), for $k=\sum_{i=1}^{r}\sum_{j=0}^{L_i-1}2^{q_i+j}$ we have
\begin{align*}
n_k(f)&=2k\left(\sum_{l=1}^{+\infty}\frac{1}{2^{a_l-1}}\right)-\sum_{i=1}^{r}\left(\sum_{u=0}^{p_i-1}\left(\sum_{j=1}^{a_{m_i-u}-a_{m_i-(u+1)}}2^{a_{m_i-u}-j+1}\right)\left(\sum_{l=m_i-u}^{+\infty}\frac{1}{2^{a_l-1}}\right)\right.&\\
&\left.+\left(\sum_{j=0}^{t_i}2^{a_{m_i}+j+1}\right)\left(\sum_{l=m_i+1}^{+\infty}\frac{1}{2^{a_l-1}}\right)
+\left(\sum_{l=1}^{s_i}2^{a_{m_i-p_i}-l+1}\right)\left(\sum_{l=m_i-p_i}^{+\infty}\frac{1}{2^{a_l-1}}\right)\right)&\\
&-2\sum_{i=1}^{r}\left(\sum_{j=a_{m_i-p_i}-s_i}^{a_{m_i}+t_i}f(j)-\sum_{j=1}^{L_i}f(j)\right)-f(L_1).&\\
\end{align*}
\end{lemma}

\begin{proof} We only prove the lemma for $t_i<L_i,$ the other case being similar but simpler. We use the notations (\ref{nota_sequences}) to write
\begin{equation}\label{equation_n2}
\sum_{j=0}^{L_i-1}n_{2^{q_i+j}}=\sum_{j=0}^{L_i-1}n_{2^{a_{m_i}+t_i-j}}=\sum_{j=0}^{t_i}n_{2^{a_{m_i}+j}}+\sum_{j=1}^{L_i-1-s_i-t_i}n_{2^{a_{m_i}-j}}+\sum_{j=1}^{s_i}n_{2^{a_{m_i-p_i}-j}}.
\end{equation}

It remains to compute these three sums. We begin by the second one, dropping for the moment the index $i$ for sake of readability. Using Lemma \ref{lemformuleamq}, 
we write
\begin{align*}
\sum_{j=1}^{L-1-s-t}n_{2^{a_{m}-j}}&=\sum_{u=0}^{p-1}\sum_{j=1}^{a_{m-u}-a_{m-(u+1)}}n_{2^{a_{m-u}-j}}\\
&=\sum_{u=0}^{p-1}\sum_{j=1}^{a_{m-u}-a_{m-(u+1)}}\left(2^{a_{m-u}-j+1}\sum_{l=1}^{m-(u+1)}\frac{1}{2^{a_l-1}}-2(m-(u+1))+1\right).
\end{align*}
Thus, we deduce
\begin{equation}\label{equation_n3}
\begin{array}{rcl}\displaystyle\sum_{j=0}^{L-1-s-t}n_{2^{a_{m}-j}}&=&
\displaystyle\sum_{u=0}^{p-1}\left(\sum_{j=1}^{a_{m-u}-a_{m-(u+1)}}2^{a_{m-u}-j+1}\right)\left(\sum_{l=1}^{m-(u+1)}\frac{1}{2^{a_l-1}}\right)\\
&&\displaystyle -2\sum_{u=1}^{p}(m-u)(a_{m-(u-1)}-a_{m-u})+L-s-t.\end{array}
\end{equation}

In the same spirit we compute the first and third sums as follows
\begin{equation}\label{equation_n4}
\sum_{j=0}^{t}n_{2^{a_m}+j}=\left(\sum_{j=0}^{t}2^{a_m+j+1}\right)\left(\sum_{l=1}^{m}\frac{1}{2^{a_l-1}}\right)-2m(t+1)+t+1
\end{equation} 
and
\begin{equation}\label{equation_n5}\sum_{j=1}^{s}n_{2^{a_{m-p}-j}}=\left(\sum_{l=1}^{s}2^{a_{m-p}-l+1}\right)\left(\sum_{l=1}^{m-(p+1)}\frac{1}{2^{a_l-1}}\right)-2(m-(p+1))s+s.
\end{equation}

Moreover since we have by definition $t_i<a_{m_{i+1}}-a_{m_i}$ and $s_i<a_{m_i-p_i}-a_{m_i-p_i-1},$ when we gather equations (\ref{equation_n3}), (\ref{equation_n4}) and (\ref{equation_n5}), we have to compute the following sum
\begin{equation}\label{equation_n6}\begin{array}{l}
\displaystyle\sum_{u=1}^{p_i}(m_i-u)(a_{m_i-(u-1)}-a_{m_i-u})+m_i(t_i+1)+(m_i-(p_i+1))s_i\\
\quad =\displaystyle\sum_{u=1}^{p_i}\sum_{j=0}^{a_{m_i-(u-1)}-a_{m_i-u}-1}f(a_{m_i-u}+j)+
\sum_{j=0}^{t_i}f(a_{m_i}+j)+\sum_{j=1}^{s_i}f(a_{m_i-p_i}-j)\\
\quad =\displaystyle\sum_{j=a_{m_i-p_i}-s_i}^{a_{m_i}+t_i}f(j).\end{array}
\end{equation}
Thus thanks to Lemma \ref{lemdecsum} and the equations (\ref{equation_n2}), (\ref{equation_n3}), (\ref{equation_n4}), 
(\ref{equation_n5}), (\ref{equation_n6}), we deduce
\begin{align*}
n_k&=\sum_{i=1}^{r}\sum_{j=0}^{L_i-1}n_{2^{q_i+j}}+2\sum_{i=1}^{r}\sum_{j=1}^{L_i}f(j)-\sum_{i=1}^{r}L_i-f(L_1)\\
&=\sum_{i=1}^{r}\left(\sum_{u=0}^{p_i-1}\left(\sum_{j=1}^{a_{m_i-u}-a_{m_i-(u+1)}}2^{a_{m_i-u}-j+1}\right)\left(\sum_{l=1}^{m_i-(u+1)}\frac{1}{2^{a_l-1}}\right)+\left(\sum_{j=0}^{t_i}2^{a_{m_i}+j+1}\right)\left(\sum_{l=1}^{m_i}\frac{1}{2^{a_l-1}}\right)\right.\\
&\left.+\left(\sum_{l=1}^{s_i}2^{a_{m_i-p_i}-l+1}\right)\left(\sum_{l=1}^{m_i-(p_i+1)}\frac{1}{2^{a_l-1}}\right)-2\sum_{j=a_{m_i-p_i}-s_i}^{a_{m_i}+t_i}f(j)+L_i\right)\\&+2\sum_{i=1}^{r}\sum_{j=1}^{L_i}f(j)-\sum_{i=1}^{r}L_i-f(L_1).
\end{align*}
We remark that a $\sum_{i=1}^{r}L_i$ comes out from the first sum and cancels the term lying in the end of the preceding equality, we also gather the sums over $f(j)$ and we get:
\begin{align*}
n_k&=\sum_{i=1}^{r}\left(\sum_{u=0}^{p_i-1}\left(\sum_{j=1}^{a_{m_i-u}-a_{m_i-(u+1)}}2^{a_{m_i-u}-j+1}\right)\left(\sum_{l=1}^{m_i-(u+1)}\frac{1}{2^{a_l-1}}\right)+\left(\sum_{j=0}^{t_i}2^{a_{m_i}+j+1}\right)\left(\sum_{l=1}^{m_i}\frac{1}{2^{a_l-1}}\right)\right.&\\
&\left.+\left(\sum_{l=1}^{s_i}2^{a_{m_i-p_i}-l+1}\right)\left(\sum_{l=1}^{m_i-(p_i+1)}\frac{1}{2^{a_l-1}}\right)\right)+2\sum_{i=1}^{r}\left(\sum_{j=1}^{L_i}f(j)-\sum_{j=a_{m_i-p_i}-s_i}^{a_{m_i}+t_i}f(j)\right)-f(L_1).&
\end{align*}
Then, we express the partial sums $\sum_{l=1}^{N}\frac{1}{2^{a_l-1}}$ as the series minus its remainder of order $N$ which yields: 
\begin{align*}
&n_k=\left(\sum_{l=1}^{+\infty}\frac{1}{2^{a_l-1}}\right)\sum_{i=1}^{r}\left(\sum_{u=0}^{p_i-1}\left(\sum_{j=1}^{a_{m_i-u}-a_{m_i-(u+1)}}2^{a_{m_i-u}-j+1}\right)+\sum_{j=0}^{t_i}2^{a_{m_i}+j+1}+\sum_{l=1}^{s_i}2^{a_{m_i-p_i}-l+1}\right)&\\
&-\sum_{i=1}^{r}\left(\sum_{u=0}^{p_i-1}\left(\sum_{j=1}^{a_{m_i-u}-a_{m_i-(u+1)}}2^{a_{m_i-u}-j+1}\right)\left(\sum_{l=m_i-u}^{+\infty}\frac{1}{2^{a_l-1}}\right)+\left(\sum_{j=0}^{t_i}2^{a_{m_i}+j+1}\right)\left(\sum_{l=m_i+1}^{+\infty}\frac{1}{2^{a_l-1}}\right)\right.&\\
&\left.+\left(\sum_{l=1}^{s_i}2^{a_{m_i-p_i}-l+1}\right)\left(\sum_{l=m_i-p_i}^{+\infty}\frac{1}{2^{a_l-1}}\right)\right)+2\sum_{i=1}^{r}\left(\sum_{j=1}^{L_i}f(j)-\sum_{j=a_{m_i-p_i}-s_i}^{a_{m_i}+t_i}f(j)\right)-f(L_1).&
\end{align*}
Now it suffices to remark that coming back to notations with $q_i$'s, then $k$ can be expressed in the following way 
$$k=\sum_{i=1}^{r}\left(\sum_{u=0}^{p_i-1}\left(\sum_{j=1}^{a_{m_i-u}-a_{m_i-(u+1)}}2^{a_{m_i-u}-j}\right)+\sum_{j=0}^{t_i}2^{a_{m_i}+j}+\sum_{l=1}^{s_i}2^{a_{m_i-p_i}-l}\right).$$ 
Thus we obtain:
\begin{align*}
&n_k=2k\left(\sum_{l=1}^{+\infty}\frac{1}{2^{a_l-1}}\right)-\sum_{i=1}^{r}\left(\sum_{u=0}^{p_i-1}\left(\sum_{j=1}^{a_{m_i-u}-a_{m_i-(u+1)}}2^{a_{m_i-u}-j+1}\right)\left(\sum_{l=m_i-u}^{+\infty}\frac{1}{2^{a_l-1}}\right)\right.&\\
&\left.+\left(\sum_{j=0}^{t_i}2^{a_{m_i}+j+1}\right)\left(\sum_{l=m_i+1}^{+\infty}\frac{1}{2^{a_l-1}}\right)
+\left(\sum_{l=1}^{s_i}2^{a_{m_i-p_i}-l+1}\right)\left(\sum_{l=m_i-p_i}^{+\infty}\frac{1}{2^{a_l-1}}\right)\right)&\\
&+2\sum_{i=1}^{r}\left(\sum_{j=1}^{L_i}f(j)-\sum_{j=a_{m_i-p_i}-s_i}^{a_{m_i}+t_i}f(j)\right)-f(L_1).&
\end{align*}

\end{proof}

Observe that if we set $(a_m)_m=(m)_m,$ then $f(j)=j$ for every integer $j$ and Lemma \ref{lemsumn} takes the following form.

\begin{lemma}\label{retrouve_lemme} In the aforementioned case, if $k=2^n+\sum_{i=0}^{n-1}\alpha_i2^i$ with $\alpha_i\in\{0;1\}$ 
for every $0\leq i\leq n-1$, then
$$n_k=4k-2\left(\sum_{i\in I_k} L_i(i+1)\right)-\delta_k,$$
where $I_k$ stands for the set of integers $i$ such that $\alpha_i$ is the first non-zero integer of a block (of consecutive non-zero coefficients) having length $L_i$ in the dyadic decomposition of $k$.
\end{lemma}

\begin{proof} Using the notations of Lemma \ref{lemsumn}, we have: $t_i=s_i=0,$ $L_i=p_i+1,$ $m_i=i+L_i-1.$ Therefore we deduce 
\begin{align*}
n_k&=4k-4\sum_{i=1}^r(p_i+1)-2\sum_{i=1}^r\left(\sum_{j=i}^{i+L_i-1}j-\frac{L_i(L_i+1)}{2}\right)-L_1\\
&=4k-4\sum_{i=1}^rL_i-\sum_{i=1}^r(2i+L_i-1)L_i+\sum_{i=1}^rL_i(L_i+1)-L_1\\
&=4k-2\sum_{i=1}^rL_i(i+1)-L_1.
\end{align*}
This last inequality gives the result since $L_1=\delta_k.$ 
\end{proof}

Lemma \ref{retrouve_lemme} is exactly Lemma \ref{lemma_b2} announced in the previous section.\\
 
Let us return to the general situation, using the notation $(n_k(f))$ again. 
From Lemma \ref{lemsumn}, we deduce the following estimate on the sequence $(n_k(f))$ for specific 
choices of functions $f.$

\begin{lemma}\label{lemestimation} Using the previous notations, assume that $a_m=2^{2^{\adots^{2^m}}},$ where 
$2$ appears $s$ times ($s\geq 1$). Then the associated function $f_s:\mathbb{N}\rightarrow \mathbb{N}$ is given by 
$f_s(j)=m,$ for $j\in\{a_m,\dots,a_{m+1}-1\},$ and the following estimate holds:
$$2k\left(\sum_{l=1}^{+\infty}\frac{1}{2^{a_l-1}}\right)-2\log_2(k)f_s(\lfloor\log_2(k)\rfloor)
-14\log_2(k)-8f_s(\lfloor\log_2(k)\rfloor)\leq n_k^{(s)}\leq 2k\left(\sum_{l=1}^{+\infty}\frac{1}{2^{a_l-1}}\right),$$
with $n_k(f_s)=n_k^{(s)}.$
\end{lemma}

\begin{proof} We need Lemma \ref{lemsumn} and its notations. The proof of the upper bound is obvious. 
For the lower bound, observe first that the subadditivity of $f_s$ implies that for $k=\sum_{i=1}^{r}\sum_{j=0}^{L_i-1}2^{q_i+j},$ 
$$\sum_{j=1}^{L_i}f_s(j)-\sum_{j=a_{m_i-p_i}-s_i}^{a_{m_i}+t_i}f_s(j)=
\sum_{j=1}^{L_i}\left(f_s(j)-f_s(a_{m_i-p_i}-s_i-1+j)\right)\geq -L_if_s(a_{m_i-p_i}-s_i).$$
In addition, since for every $u\geq 1$, we have $a_{u}+2<a_{u+2}$ and $\sum_{l=q}^{+\infty}2^{-j}= 2^{1-q}$, we obtain:
\begin{align*}
\left(\sum_{j=0}^{t_i}2^{a_{m_i}+j+1}\right)\left(\sum_{l=m_i+1}^{+\infty}\frac{1}{2^{a_l-1}}\right)&=4\left(\sum_{j=0}^{t_i}2^{a_{m_i}+j}\right)\left(\sum_{l=m_i+1}^{+\infty}\frac{1}{2^{a_l}}\right)&\\
&\leq 4\left(\sum_{j=0}^{t_i}2^{a_{m_i}+j}\right)\left(\frac{1}{2^{a_{m_i+1}}}+\frac{1}{2^{a_{m_i+1}+1}}+\frac{1}{2^{a_{m_i+1}+2}}\right)&\\
&\leq \frac{7}{2^{a_{m_i+1}-a_{m_i}}}\left(2^{t_i+1}-1\right)&\\
&\leq \frac{7}{2^{a_{m_i+1}-(a_{m_i}+t_i+1)}}&\\
&\leq7.&
\end{align*}
In the same spirit, we also have
\begin{align*}
\left(\sum_{l=1}^{s_i}2^{a_{m_i-p_i}-l+1}\right)\left(\sum_{l=m_i-p_i}^{+\infty}\frac{1}{2^{a_l-1}}\right)&=4\left(\sum_{l=1}^{s_i}2^{a_{m_i-p_i}-l}\right)\left(\sum_{l=m_i-p_i}^{+\infty}\frac{1}{2^{a_l}}\right)&\\
&\leq 4\left(\sum_{l=1}^{s_i}2^{a_{m_i-p_i}-l}\right)\left(\frac{1}{2^{a_{m_i-p_i}}}+\frac{1}{2^{a_{m_i-p_i}+1}}+\frac{1}{2^{a_{m_i-p_i}+2}}\right)&\\
&\leq 7\left(\sum_{l=1}^{s_i}2^{-l}\right)&\\
&\leq 7.&
\end{align*}
The same method gives again
$$\sum_{u=0}^{p_i-1}\left(\sum_{j=1}^{a_{m_i-u}-a_{m_i-(u+1)}}2^{a_{m_i-u}-j+1}\right)\left(\sum_{l=m_i-u}^{+\infty}\frac{1}{2^{a_l-1}}\right)\leq7\sum_{u=0}^{p_i-1}\left(\sum_{j=1}^{a_{m_i-u}-a_{m_i-(u+1)}}2^{-j}\right)\leq 7p_i.$$
Finally we gather these estimates and we obtain 
\begin{align*}
n_k^{(s)}&\geq 2k\left(\sum_{l=1}^{+\infty}\frac{1}{2^{a_l-1}}\right)-
7\sum_{i=1}^{r}(2+p_i)-2\sum_{i=1}^{r}L_if_s(a_{m_i-p_i}-s_i)-f_s(L_1)&\\
&\geq 2k\left(\sum_{l=1}^{+\infty}\frac{1}{2^{a_l-1}}\right)-7(2r+m_r)  -2\left(\sum_{i=1}^{r}L_i\right)f_s(a_{m_r}+t_r)-f_s(L_1).
\end{align*}
Using the fact that $a_{m_r}+t_i=q_r+L_r-1\leq \log_2(k)<q_r+L_r$ and $k=\sum_{i=1}^{r}\sum_{j=0}^{L_i-1}2^{q_i+j},$ we get
\begin{align*}
n_k^{(s)}&\geq 2k\left(\sum_{l=1}^{+\infty}\frac{1}{2^{a_l-1}}\right)-7(2\log_2(k)+f_s(\lfloor\log_2(k)\rfloor))
-2\log_2(k)f_s(\lfloor\log_2(k)\rfloor)-f_s(\lfloor\log_2(k)\rfloor)&\\
&\geq 2k\left(\sum_{l=1}^{+\infty}\frac{1}{2^{a_l-1}}\right)-2\log_2(k)f_s(\lfloor\log_2(k)\rfloor) -14\log_2(k)-
8f_s(\lfloor\log_2(k)\rfloor).
\end{align*}
This finishes the proof.
\end{proof}

We now prove that the sequence $(n_k^{(s)})$ constructed above not only has positive lower density but has also positive lower 
$\tilde{B}_s$-density, for every $s\geq 2.$

\begin{lemma}\label{dens_tilde} We have $\underline{d}_{\tilde{B}_s}((n_k^{(s)}))>0.$
\end{lemma}
\begin{proof} According to (\ref{examplelogdensity7}) from Example \ref{examplelogdensity}, we write
$$\underline{d}_{B_s}(n_k^{(s)})=
\liminf_{k\rightarrow +\infty}\left(
\frac{\sum_{j=1}^{k}e^{n_j^{(s)}/h_s(n_j^{(s)})}}
{h_s(n_k^{(s)})e^{n_k^{(s)}/h_s(n_k^{(s)})}
}\right).$$
 
Observe that Lemma \ref{lemestimation} ensures the existence of two constants $C_1,C_2>1$ such that for $N$ large enough, 
$$\frac{\sum_{j=N}^{k}e^{n_j^{(s)}/h_s(n_j^{(s)})}}
{h_s(n_k^{(s)})e^{n_k^{(s)}/
h_s(n_k^{(s)})}}
\geq 
\frac{\sum_{j=N}^{k}e^{(C_1j-C_2h_s(j))/h_s(C_1j-C_2h_s(j))}}
{h_s(C_1 k)e^{C_1 k/h_s(C_1 k)}}.$$

A summation by parts gives
$$\sum_{j=N}^{k}e^{(C_1j-C_2h_s(j))/h_s(C_1j-C_2h_s(j))}\sim 
\frac{h_s(k)}{C_1} e^{(C_1k-C_2h_s(k))/h_s(C_1k-C_2h_s(k))},\hbox{ as }k\rightarrow +\infty.$$ 
Then, a similar computation as those needed for (\ref{examplelogdensity7}) from Example \ref{examplelogdensity} leads to the following estimate 
$$\frac{\sum_{j=N}^{k}e^{(C_1j-C_2h_s(j))/h_s(C_1j-C_2h_s(j))}}
{h_s(C_1 k)e^{C_1 k/h_s(C_1 k)}}\sim \frac{e^{-C_2}}{C_1},\hbox{ as }k\rightarrow +\infty,$$ 
which gives the desired conclusion.
\end{proof}

This allows to prove the following combinatorial lemma which extends Lemma \ref{Lemgrope}.

\begin{lemma}\label{Lemgropef} There exist pairwise disjoint subsets $B^{(s)}(l,\nu),$ $l,\nu\geq 1,$ of 
$\mathbb{N}$ having positive $\tilde{B}_s$ density such that, for any $n\in B^{(s)}(l,\nu)$ and $m\in B^{(s)}(k,\mu)$, we 
have that $n\geq f_l(\nu)$ and 
$$\vert n-m\vert\geq f_l(\nu)+f_l(\mu)\hbox{ if }n\neq m.$$
\end{lemma}

\begin{proof} We consider the sequence $(n_{k}^{(s)})$ constructed above and also sets $I(l,\nu)$ constructed in \cite{Grope} 
that we recalled just after Lemma \ref{Lemgrope}. We also define $B^{(s)}(l,\nu):=\{n_{k}^{(s)}; k\in I(l,\nu)\}.$ These sets are clearly pairwise disjoint 
since the sets $I(l,\nu)$ are, and the sequence is $(n_{k}^{(s)})$ increasing. Moreover by definition of 
the sets $I(l,\nu)$, that are arithmetic sequences, and Lemma \ref{lemestimation}, the conclusion of Lemma \ref{dens_tilde} remains true, 
i.e. the sets $B^{(s)}(l,\nu)$ 
have positive lower $\tilde{B}_s$-density. 
Then by definition of $n_{k}^{(s)}$ from (\ref{Eqnkmod}), we get $n_{k}^{(s)}\geq f_s(\delta_k)=f_s(\nu)$.
Finally, if $n_{j}^{(s)}\in B^{(s)}(l,\nu)$ and $n_{m}^{(s)}\in B^{(s)}(k,\mu)$ with $j>m$, then
$$n_{j}^{(s)}-n_{m}^{(s)}=f(\delta_m)+2\sum_{i=m+1}^{j-1}f(\delta_i)+f(\delta_j)\geq f(\mu)+f(\nu).$$
\end{proof}

This strengthened version of Lemma \ref{Lemgrope} allows us to give a stronger conclusion to the so-called Frequent Hypercyclicity Criterion whose proof will be only sketched since it is an adaptation of the classical proof given in \cite[Theorem 9.9]{Grope}. 

\begin{theorem}\label{main_theorem} Let $T$ be an operator on a separable Fr\'echet space $X.$ If there is a dense 
subset $X_0$ of $X$ and a map $S:X_0\rightarrow X_0$ such that, for any $x\in X_0,$ 
\begin{enumerate}[(i)]
\item $\displaystyle\sum_{n=0}^{\infty}T^nx$ converges unconditionally,
\item $\displaystyle\sum_{n=0}^{\infty}S^nx$ converges unconditionally,
\item $TSx=x,$
\end{enumerate}
then $T$ is $\tilde{B}_s$-frequently hypercyclic, for every $s\geq 2.$ 
\end{theorem}

\begin{proof} Let $(y_n)$ be a dense sequence from $X_0$ that is dense in $X$. Let $\Vert .\Vert$ denote an $F$-norm that defines the topology of $X.$ 
The unconditional convergence of the series $(ii)$ and $(iii)$ allows to find, for every $l\in\N$, an integer $N_l\geq 1$ such that for every $j\leq l$ and every finite set $F\subset \{N_l;N_l+1;\ldots\}$,
$$\left\Vert \sum_{n\in F}T^{n}y_l\right\Vert\leq \frac{1}{l2^l}\text{ and }\left\Vert \sum_{n\in F}S^{n}y_l\right\Vert\leq \frac{1}{l2^l}.$$
Now let $(M_l)$ be an increasing sequence such that $f_s(M_l)\geq N_l$ and $(f_s(M_l))$ is increasing. We also define
$$B^{(s)}:=\bigcup_{l=1}^{+\infty}B^{(s)}(l,M_l)$$
and 
$$z_n=y_l\text{ if } n\in B^{(s)}(l,M_l).$$
Finally we claim that 
$$x=\sum_{n\in\N}S^n(z_n)$$
defines a $\tilde{B}_s$-frequently hypercyclic vector for $T$.
From this point, the proof is just an adaptation of the proof of the Frequent Hypercyclicity Criterion from \cite{Grope} replacing Lemma \ref{Lemgrope} by Lemma \ref{Lemgropef} stated above.
\end{proof}

We may also deduce the following corollary using Lemma \ref{lemmacomp}.

\begin{corollary}
Under the assumptions of the previous proposition, the operator $T$ is $B_r$-frequently hypercyclic for every $r>1$.
\end{corollary}

\section{A frequently hypercyclic operator which is not $A_r$-frequently hypercyclic}\label{fhc_op} 
In this final section, we are going to show that there exists frequently hypercyclic operator, that 
do not belong to the class of $A_r$-frequently hypercyclic operator, for any $0<r\leq 1.$ According to Proposition 
\ref{main_prop} or Theorem \ref{main_theorem} such an operator 
cannot satisfy the Frequent Hypercyclicity Criterion. To build it, we are going to use 
several ideas of the work \cite{Bayru}, where the authors provide some counterexamples to
questions regarding frequent hypercyclicity.\\

In a recent paper, Bayart and Ruzsa gave a characterization of frequently hypercyclic weighted 
shifts on the sequence spaces $\ell^p$ and $c_0$. We recall here their result on $c_0(\N)$ that will be useful in 
the following \cite[Therorem 13]{Bayru}.

\begin{theorem}\label{TheoBay}
Let $w=(\omega_n)_{n\in\N}$ be a bounded sequence of positive integers. Then $B_w$ is frequently hypercyclic on $c_0(\N)$ if and only if there exist a sequence $(M(p))$ of positive real numbers tending to $+\infty$ and a sequence $(E_p)$ of subsets of $\N$ such that:
\begin{enumerate}[(a)]
\item For any $p\geq 1$, $\underline{d}\left(E_p\right)>0$;
\item For any $p,q\geq1$, $p\neq q$, $\left(E_p+[0,p]\right)\cap\left(E_q+[0,q]\right)=\emptyset$;
\item $\lim_{n\to\infty,\ n\in E_p+[0,p]} \omega_1\cdots \omega_n=+\infty$;
\item For any $p,q\geq1$, for any $n\in E_p$ and any $m\in E_q$ with $m>n$, for any $t\in\{0,\ldots,q\}$,
$$\omega_1\cdots \omega_{m-n+t}\geq M(p)M(q).$$
\end{enumerate}
\end{theorem}

In the same paper, the authors also provide examples of a $\mathcal{U}$-frequently hypercyclic weighted shift which is not frequently hypercyclic and of a frequently hypercyclic weighted shift which is not distributionally chaotic. In what follows, we modify these constructions in order to provide a frequently hypercyclic weighted shift on $c_0(\N)$ which is not $A_r$-frequently hypercyclic for any $0<r<1$.
To that purpose, we will need the following lemma \cite[Lemma 1]{Bayru}:

\begin{lemma} There exist 
$a>1$ and $\varepsilon>0$ such that  
$\overline{d}\left(\cup_{u\geq1}I_{u}^{a,4\varepsilon}\right)<1$ and, for any integer $u>v\geq  1$,
$$I_{u}^{a,2\varepsilon}\cap I_{v}^{a,2\varepsilon}=\emptyset, \ I_{u}^{a,2\varepsilon}-I_{v}^{a,2\varepsilon}\subset I_{u}^{a,4\varepsilon},$$
where $I_u^{a,\varepsilon}=[(1-\varepsilon)a^u,(1+\varepsilon)a^u].$
\end{lemma}

The philosophy of the previous lemma is that it suffices to choose $a$ very large and at the same time $\varepsilon$ very small to obtain the result stated. This allows us to strengthen this lemma demanding also that the following 
condition holds:
\begin{equation}\label{cond1}
\frac{1-\varepsilon}{1+\varepsilon}a>1.
\end{equation}

From now on, we suppose that $a$ and $\varepsilon$ are given by the previous lemma with the additional 
condition $(\ref{cond1}).$

Let also $(b_p)$ be an increasing sequence of integers such that
\begin{equation}\label{cond2}
\sum_{q\geq 1}\frac{(4q+1)(2q+1)}{b_q}e^{2q}<\infty\text{ and }  b_p\geq 8p.
\end{equation}

Finally, let $(A_p)$ be any syndetic partition of $\N$ and 
$$E_p=\cup_{u\in A_p}\left(I_{u}^{a,\varepsilon}\cap\left(b_p\N+[0,p]\right)\right).$$

Bayart and Ruzsa construct such sets and they prove that these sets have positive lower density 
\cite[Lemma 2]{Bayru}. Then, for the same reasons we have $\underline{d}(E_p)>0.$ 
Further, the following lemma is almost the same as \cite[Lemma 3]{Bayru} once again and it still holds 
in our context:

\begin{lemma}
Let $p,q\geq1$, $n\in E_p$, $m\in E_q$ with $n\neq m$. Then $\vert n-m\vert>\max(p,q)$.
\end{lemma}

In particular, $(E_p+[0,p])\cap(E_q+[0,q])=\emptyset$ if $p\neq q$.\\
Thus, the sequence of sets $(E_p)$ satisfy conditions $(a)$ and $(b)$ from Theorem \ref{TheoBay}.

We now turn to the construction of the weights of the weighted shift we are looking for. For this construction, we also draw our inspiration from constructions made in \cite{Bayru}. We set:
$$w_{0}^{p}\cdots w_{k-1}^{p}=\begin{cases}1 &\text{ if }k\notin b_p\N+[-4p,4p]\\
2^p &\text{ if }k\in b_p\N+[-2p,2p]
\end{cases}$$
and for every $k\in\N$, $\frac{1}{2}\leq w_{k}^{p}\leq 2$.
Then for $p,q\geq1$, $u\in A_p$ and $v\in A_q$ with  $u>v$ we define
$$w_{0}^{u,v}\cdots w_{k-1}^{u,v}=\begin{cases}1 &\text{ if }k\notin I_{u}^{a,4\varepsilon}\\
\max(2^p,2^q) &\text{ if }k\in I_{u}^{a,\varepsilon}-I_{v}^{a,\varepsilon}+[0,p]
\end{cases}$$
and for every $k\in\N$, $\frac{1}{2}\leq w_{k}^{u,v}\leq 2$.

We are now able to give the definition of the weight $w$. This one is constructed in order to satisfy the following equality:
$$w_{0}\cdots w_{n-1}=\max_{p,u,v}\left(w_{0}^{p}\cdots w_{n-1}^{p},w_{0}^{u,v}\cdots w_{n-1}^{u,v}\right).$$
It is clear by construction that for every $n\in\N$, $\frac{1}{2}\leq w_n\leq2$, so the weighted backward shift $B_w$ is bounded and invertible. Moreover this construction satisfies condition $(c)$ in Theorem \ref{TheoBay}.

Since we want to prove that $B_w$ is frequently hypercyclic, the only condition left to prove is condition $(d)$ from Theorem \ref{TheoBay}.
Thus let $p,q\geq1$, $n\in E_p$ and $m\in E_q$ with $m>n$ and $t\in[0,q]$. Then we have two cases:
\medskip

$\bullet$ If $p=q$, then $m-n+t
\in b_q\N+[-q,2q]$ and the definition of $w$ ensures that $w_{0}\cdots w_{m-n+t}\geq 2^{q}$.

\medskip
$\bullet$ If $p\neq q$, then there exists $u>v$ such that $n\in I_{v}^{a,\varepsilon}$ and $m\in I_{u}^{a,\varepsilon}$. Thus, by definition of $w$,
$$w_{0}\cdots w_{m-n+t}\geq \max\left(2^p;2^q\right)\geq 2^{\frac{p+q}{2}}\geq \lfloor2^{\frac{p}{2}}\rfloor\cdot\lfloor2^{\frac{q}{2}}\rfloor.$$

Now, one may define $M(p):=\lfloor2^{\frac{p}{2}}\rfloor$ and each case above satisfies condition $(d)$ from Theorem \ref{TheoBay}. Thus we have proved that the weighted shift $B_w$ is frequently hypercyclic.\\

We now turn to the $A_r$-frequent hypercyclicity of $B_w$. We are going to prove by contradiction that $B_w$ is not $A_r$-frequently hypercyclic for every $0<r<1$.\\
Let us suppose that $B_w$ is $A_r$-frequently hypercyclic and that $E\subset \N$ is such that $\underline{d}_{A_r}(E)>0$ and $\lim_{n\to\infty, n\in E} w_1\cdots w_n=+\infty$. Such a set exists since $B_w$ is $A_r$-frequently hypercyclic. Indeed it suffices to consider $E=\{n\in\N: \Vert B_{w}^{n}(x)-e_0\Vert\leq \frac{1}{2}\}$ where $x$ is a $A_r$-frequently hypercyclic vector.

For every $p\geq 1$, we consider the set:
$$F_p=\{n\in E: w_1\cdots w_n>2^p\}.$$
This set is a cofinite subset of $E$, so it has the same lower $A_r$-density.
We also consider an increasing enumeration $(n_k)$ of $A_p$.

Then $$\underline{d}_{A_r}(F_p)\leq \liminf_{k\to\infty}\left(
\displaystyle\frac{\displaystyle\sum_{n\leq (1+\varepsilon)a^{n_k},\atop  n\in F_p}e^{n^r}+
\sum_{(1+\varepsilon)a^{n_{k}}< n\leq (1-\varepsilon)a^{n_{k+1}},\atop  n\in \cup_{q>p}\left(b_q\N+[-2q,2q]\right)}e^{n^r}}
{\displaystyle\sum_{n\leq (1-\varepsilon)a^{n_{k+1}}}e^{n^r}}\right).$$

Moreover, since we have $\sum_{n\leq N}e^{n^r}\sim\frac{1}{r}N^{1-r}e^{N^r}$, as 
$N$ tends to $\infty,$ we get
$$\underline{d}_{A_r}(F_p)\leq \liminf_{k\to\infty}\left(
\displaystyle
\frac{\frac{1}{r}\left((1+\varepsilon)a^{n_k}\right)^{1-r}e^{\left((1+\varepsilon)a^{n_k}\right)^r}+
\displaystyle\sum_{(1+\varepsilon)a^{n_{k}}< n\leq (1-\varepsilon)a^{n_{k+1}},\atop  n\in \cup_{q>p}\left(b_q\N+[-2q,2q]\right)}e^{n^r}}
{\frac{1}{r}\left((1-\varepsilon)a^{n_{k+1}}\right)^{1-r}e^{\left((1-\varepsilon)a^{n_{k+1}}\right)^r}}\right).$$
A straightforward computation using inequality (\ref{cond1}) proves that the first term on the right-hand side tends to $0.$ We now focus on the second term:
\begin{align*}
\underline{d}_{A_r}(F_p)&\leq \liminf_{k\to\infty}\left(
\frac{\displaystyle\sum_{(1+\varepsilon)a^{n_{k}}< n\leq (1-\varepsilon)a^{n_{k+1}},\atop n\in \cup_{q>p}\left(b_q\N+[-2q,2q]\right)}e^{n^r}}{\frac{1}{r}\left((1-\varepsilon)a^{n_{k+1}}\right)^{1-r}e^{\left((1-\varepsilon)a^{n_{k+1}}\right)^r}}\right)\\
&\leq\liminf_{k\to\infty}\frac{\sum_{q>p}(4q+1)\sum_{j=1}^{\frac{(1-\varepsilon)a^{n_{k+1}}}{b_q+2q}}e^{\left(jb_q+2q\right)^r}}{\frac{1}{r}\left((1-\varepsilon)a^{n_{k+1}}\right)^{1-r}e^{\left((1-\varepsilon)a^{n_{k+1}}\right)^r}}.\\
\end{align*}
An classical calculation ensures that we have 
$$\sum_{j=1}^{\frac{(1-\varepsilon)a^{n_{k+1}}}{b_q+2q}}e^{\left(jb_q+2q\right)^r}\sim
\frac{\left(\frac{b_q(1-\varepsilon) a^{n_{k+1}}}{b_q+2q}+2q\right)^{1-r}}{r b_q}e^{\left(\frac{b_q(1-\varepsilon) a^{n_{k+1}}}{b_q+2q}+2q\right)^r},$$ as 
$k$ tends to $\infty.$ Thus we obtain
\begin{align*}
\underline{d}_{A_r}(F_p)&\leq\liminf_{k\to\infty}\frac{\sum_{q>p}(4q+1)\left(\frac{b_q(1-\varepsilon) a^{n_{k+1}}}{b_q+2q}+2q\right)^{1-r}e^{\left(\frac{b_q(1-\varepsilon) a^{n_{k+1}}}{b_q+2q}+2q\right)^r}}{b_q\left((1-\varepsilon)a^{n_{k+1}}\right)^{1-r}e^{\left((1-\varepsilon)a^{n_{k+1}}\right)^r}}\\
&\leq \liminf_{k\to\infty}\sum_{q>p}\frac{4q+1}{b_q}\left(\frac{b_q}{b_q+2q}+\frac{2q}{(1-\varepsilon) a^{n_{k+1}}}\right)^{1-r}e^{\left((1-\varepsilon) a^{n_{k+1}}+2q\right)^r-\left((1-\varepsilon)a^{n_{k+1}}\right)^r}\\
&\leq\sum_{q>p}\frac{4q+1}{b_q}(1+2q)^{1-r}e^{(2q)^r}\\
&\leq\sum_{q>p}\frac{4q+1}{b_q}(1+2q)e^{2q}.
\end{align*}
Recall that this does not require any property on $p$ so we can let $p$ tend to infinity which, thanks to (\ref{cond2}), implies that $\underline{d}_{A_r}(E)=\lim_{p\to\infty} \underline{d}_{A_r}(F_p)=0$, hence we obtain a contradiction.
Thus the weighted shift $B_w$ is not $A_r$-frequently hypercyclic. From this construction together with Corollary \ref{coro_ar}, we deduce the following 
result.

\begin{theorem}\label{counterexample} There exists a frequently hypercyclic operator being not $A_r$-frequently hypercyclic, for 
any $0<r\leq 1,$ hence which does not satisfy the Frequent Hypercyclicity Criterion.
\end{theorem}

\end{document}